\documentclass[twoside,leqno]{article}
\usepackage[letterpaper]{geometry}
\usepackage{siamproceedings}

\usepackage{xcolor}

\newcommand{\iso}{\xrightarrow{\simeq}}
\usepackage[T1]{fontenc}
\usepackage{amsfonts}
\usepackage{amssymb,amsmath,epsfig,amscd,stmaryrd}
\usepackage{graphicx}
\usepackage{epstopdf}
\usepackage{enumitem}
\usepackage{algorithmic}
\ifpdf
  \DeclareGraphicsExtensions{.eps,.pdf,.png,.jpg}
\else
  \DeclareGraphicsExtensions{.eps}
\fi

\newsiamremark{remark}{Remark}
\newsiamremark{example}{Example}
\newsiamremark{hypothesis}{Hypothesis}
\crefname{hypothesis}{Hypothesis}{Hypotheses}
\newsiamthm{claim}{Claim}
\usepackage{amsopn}

\renewcommand\mathbb{\mathbf}
\renewcommand\setminus{\smallsetminus}
\def\Omegacirc{\Omega^{\mathrm{circ}}}
\def\Omegalune{\Omega^{\mathrm{lune}}}
\def\mv{\mu_{\mathrm{Haar}}}
\def\walpha{\widetilde{\alpha}}
\def\wbeta{\widetilde{\beta}}
\def\wgamma{\widetilde{\gamma}}
\def\NwithoutzeroA{\mathbf{N}_{> 0}}
\def\N{\mathbf{N}}
\def\T{\mathbf{T}}
\def\R{\mathbf{R}}
\def\Z{\mathbf{Z}}
\def\Q{\mathbf{Q}}
\def\C{\mathbf{C}}

\def\D{\mathbf{D}}
\def\Db{\overline{\mathbf{D}}}
\def\bb{\mathbf{b}}
\def\be{\mathbf{e}}
\def\bP{\mathbf{P}}
\def\cV{\mathcal{V}}
\def\cO{\mathcal{O}}
\def\Gm{\mathbb{G}_m}
\def\WP{\mathcal{W}^{+}}

\def\rank{\mathrm{rank} \,}
\def\ardeg{\widehat{\deg} \,}

\def\ovE{\overline{E}}
\def\cL{\mathcal{L}}
\def\ovcL{\overline{\mathcal{L}}}
\def\Qb{\overline{\Q}}
\def\mueff{\mu_{\mathrm{eff}}}

\def\mueffKu{\mu_{\mathrm{eff},K,u}}
\def\mueffKv{\mu_{\mathrm{eff},K,v}}
\def\P{\mathbb{P}}

\def\Gal{\mathrm{Gal}}
\def\PGL{\mathrm{PGL}}
\def\hM{\mathcal{M}}

\def\MT{\mathcal{MT}}
\def\PP{\mathcal{P}}

\def\BdR{B^{+}_{\mathrm{dR}}}

\def\LL{\mathcal{L}}
\def\hO{\mathcal{O}}

\DeclareMathOperator{\Spec}{Spec}

\newmuskip\pFqmuskip

\newcommand*\pFq[6][8]{
  \begingroup 
  \pFqmuskip=#1mu\relax
  \mathcode`\,=\string"8000
  \begingroup\lccode`\~=`\,
  \lowercase{\endgroup\let~}\pFqcomma
  {}_{#2}F_{#3}{\left[\genfrac..{0pt}{}{#4}{#5};#6\right]}%
  \endgroup
}
\newcommand{\pFqcomma}{\mskip\pFqmuskip}

\begin{document}

\newcommand\relatedversion{}

\title{\Large Arithmetic holonomy bounds and effective Diophantine approximation\thanks{2020 Mathematics Subject Classification. 11J68, 11J72}}
    \author{Frank Calegari\thanks{The University of Chicago,
5734 S University Ave,
Chicago, IL 60637, USA.
Supported in part by NSF Grants DMS-2001097 and DMS-2450123.  (\email{fcale@math.uchicago.edu})}, Vesselin Dimitrov\thanks{Department of Mathematics, 
 California Institute of Technology,
Pasadena, CA 91125, USA. 
 (\email{dimitrov@caltech.edu})}, and Yunqing Tang\thanks{Department of Mathematics, University of California, Berkeley, Evans Hall, Berkeley, CA 94720, USA. Supported in part by NSF Grant DMS-2231958 and a Sloan Research Fellowship.
  (\email{yunqing.tang@berkeley.edu}).}}

\date{}

\maketitle

\fancyfoot[R]{\scriptsize{Copyright \textcopyright\ 2026 by SIAM\\
Unauthorized reproduction of this article is prohibited}}

\begin{abstract}
In this paper, we explore several threads arising from our recent joint work on arithmetic holonomy bounds, which were originally devised  to  prove new irrationality results based on the method of Ap\'ery limits.
We propose a new method to address effective Diophantine approximation on the projective line and the multiplicative group. 
This method, and all our other results in the paper, emerged from quantifying our holonomy bounds in a way that directly yields effective measures
of irrationality and linear independence. Applying these to a dihedral algebraic construction, we derive good effective irrationality measures
for high order roots of an algebraic number, in an approach that might 
 be considered 
a multivalent continuation of the classical hypergeometric method
of Thue, Siegel, and Baker. A well-known Dirichlet approximation argument of Bombieri allows one to derive from this the classical effective Diophantine theorems, 
hitherto only approachable by Baker's linear forms in logarithms or by Bombieri's equivariant Thue--Siegel method. These include the algorithmic resolution of the two-variable
$S$-unit equation, the Thue--Mahler equation, and the hyperelliptic and superelliptic equations, as well as the Baker--Feldman effective power sharpening of Liouville's theorem. We also 
give some other applications, including 
irrationality measures for
 the classical $L(2,\chi_{-3})$ and the $2$-adic~$\zeta(5)$, and a new proof of the transcendence of~$\pi$. 
Due to space limitations, a full development of these ideas will be deferred to future work.\end{abstract}

\section{Effective Diophantine approximation on~$\Gm$.} \label{effectiveintro}
Denote by~$h : \Gm(\Qb) \to [0,\infty)$ the canonical logarithmic absolute Weil height on the multiplicative group of the field of algebraic numbers, 
and by $H := \exp(h)$ its exponentiated (``multiplicative'') version. We refer to Bombieri and Gubler's book~\cite[\S~1.5.7]{BombieriGubler} for the definitions and
basic properties of these notions, and in particular~\cite[\S~1.4.3]{BombieriGubler} for the normalization $|\cdot|_v$ we will be using of the absolute value in the place $v \in M_K$
of a number field~$K$. These are the normalizations (see also \cref{notation}) such
that, with  $\log^+(t) := \max\left(0, \log{|t|} \right)$
and $\alpha \in \Qb$ is arbitrary, we have
 $h(\alpha) = \sum_{v \in M_K} \log^+{|\alpha|_v}$ independently of a choice of number 
 field $K$ 
 containing~$\alpha$.

A landmark theorem in Diophantine analysis takes on the following minimalistic but structurally robust form. 

\begin{theorem}  \label{main Diophantine theorem}
Consider the following data: 
\begin{itemize}
\item a number field~$K$;
\item a place~$v$ of~$K$;  
\item a finitely generated subgroup $\Gamma < K^{\times}$  of the multiplicative group $\Gm(K) = K^{\times}$; 
\item a positive number $\varepsilon > 0$. 
\end{itemize}
From these data, one can extract an effectively computable function $C(K, v,  \Gamma, \varepsilon) \in \R$, for which
the following Diophantine boundedness property takes place: 
For every $A \in K^{\times}$, all solutions $\gamma \in \Gamma$ of the Diophantine inequality
\begin{equation}   \label{main Diophantine inequality}
| 1 - A \gamma |_v \leq H(\gamma)^{-\varepsilon}
\end{equation}
satisfy the effective height bound
\begin{equation}  \label{effective height bound}
h(\gamma)  \leq  C(K, v, \Gamma,\varepsilon) \left( 1 + h(A) \right). 
\end{equation}
\end{theorem}

One  goal of this paper is to outline a new  proof of~\cref{main Diophantine theorem}.
We begin, however,  with an abridged history.
\cref{main Diophantine theorem} was first proved by Alan Baker~\cite{BakerSharpeningII} in the Archimedean case 
(enhancing his original method for lower bounds on linear forms in logarithms in many variables through ideas of Stark and Feldman),
and Kunrui Yu~\cite{YuII} in the nonarchimedean case. 
The effectivity clause is the important one---it is the feature allowing for an algorithmic output of 
all solutions to many Diophantine equations and inequalities
(in the spirit of Hilbert's Tenth Problem).
In contrast, the mere existence of some $C(K,v,\Gamma,\varepsilon) \in \R$ follows  
easily from Roth's (or indeed from Siegel's) ineffective Diophantine exponent of an arbitrary algebraic number, as Gelfond observed based on Siegel's argument with collecting coset representatives under a high degree isogeny $[r] : \Gm \to \Gm$. (See, for example, \cite[Theorem 1.9]{WaldschmidtBook} for how the $A=1$ case follows ineffectively from Roth's theorem. The general argument for~$A \in K^{\times}$ is similar
and implicit in this reduction scheme, see also~\cref{the new proof} below.)

By a standard and easy argument (see~\cite[\S~5.4.3]{BombieriGubler}),  the special case~$A=1$ of~\cref{main Diophantine theorem} (formulated
as Theorem~5.4.1 in~\cite{BombieriGubler}) provides an effective height bound on the general $S$-unit equation $x+y = 1$ in two
variables, to be solved in the group of $S$-units $x,y \in \cO_{K,S}^{\times}$ in a number field~$K$, with $S \subset M_K$ signifying an arbitrary finite set of places
that contains all Archimedean places, namely:
\begin{equation}  \label{S-unit bound}
h(x) + h(y) \leq \log{16} + 2 \max_{v \in S} \left\{ C\left(K, v, \cO_{K,S}^{\times},  \frac{1}{2|S|} \right) \right\}. 
\end{equation}
In turn, by the classical arguments of Siegel and Mahler (see \cite[\S~5.3]{BombieriGubler} or~\cite[\S~9.6]{EvertseGyory}),
an effective height bound~\cref{S-unit bound} leads in theory to an algorithmic
resolution of the general Thue--Mahler $F(x,y) \in \cO_{K,S}^{\times}$ and superelliptic $y^m = f(x)$ equations in $S$-integers $x,y \in \cO_{K,S}$, and more broadly to compute the set of integral points on various other (although far from all) affine algebraic curves. Here, $F(x,y)
\in K[x,y]$ is an arbitrary homogeneous form and $f(x) \in K[x]$ is an arbitrary polynomial. 
Furthermore, the linear dependence in~\Cref{effective height bound} on the parameter~$h(A)$, which in the Archimedean case
is formulated and proved as Baker's Theorem~1.10 in~\cite[\S\S~1, 4, 11, E]{BugeaudBook}, also allows one (via the Thue equation, see~\cite[\S\S~4.3, 4.4]{BugeaudBook}) 
 to 
derive a form of the Baker--Feldman theorem~\cite{FeldmanL} on the effective improvement $\mueff(\alpha) \leq \deg{\alpha} - \epsilon\left( \Q(\alpha) \right)$
of Liouville's Diophantine exponent~$d = \deg{\alpha}$ for an arbitrary algebraic number of a degree~$d \geq 3$. 

 The best (known) dependence 
\begin{equation} \label{depend C}
C(K,v,\Gamma,\varepsilon) \ll_{[K:\Q], v, \varepsilon,t} \prod_{i=1}^{t} (1+h(\xi_i))
\end{equation}
 on the heights of a set of generators~$\xi_1, \ldots, \xi_t$ of~$\Gamma$ (apart from an absolute
numerical coefficient involved, which however is frequently quite important in practical applications) comes from Baker's theory of linear forms in logarithms in many variables, cf. \cite{BakerWustholzMain} or~\cite{Waldschmidt1993,WaldschmidtBook} in the Archimedean case, and~\cite{YupadicI,YupadicII,YupadicIII} in the nonarchimedean case. Note that
the ultimate Baker--W\"ustholz dependency in the form~\cref{depend C} 
eliminates
the special role of the coset~$A\Gamma$, and~$A = 1$ is no longer any loss of generality in
this formulation.
 The form~\cref{depend C} has no reason to have any finality; 
 the most optimistic  folklore conjecture closely related to $abc$ speculates that the product in~\cref{depend C} should get replaced by a sum.
For simplicity of the exposition, we shall hence stick to the minimalistic essential form of~\cref{main Diophantine theorem}, including the distinguished coset parameter~$A$
in the shape of Baker's {\it Sharpening~II}~\cite{BakerSharpeningII}, and we
abstain from pursuing here the refinement to the form~\cref{depend C}. 
A number of different and easier proofs~\cite{BiluBugeaud,BugeaudSunit}, \cite[\S~4.1]{BugeaudBook}, \cite[Corollary 10.18]{WaldschmidtBook} 
of ~\cref{main Diophantine theorem}
were obtained using
 Bombieri's geometry of numbers argument. 
 The most rudimentary form of Bombieri's idea
  \cite[Lemma~4.2]{BugeaudBook} goes back to an observation of Stark~\cite[page~262]{StarkDirichlet}
  which we will use ourselves in~\cref{the new proof}. 
  This argument only requires relatively simpler-to-prove bounds~\cite{Laurent1994,LMN,BugeaudLaurent} for two logarithms, and
  in particular it does not require the optimal (logarithmic) dependence on the variables of the linear form (or the exponents, as one frequently says). What they \emph{do} require,
  however, is that they  come with Waldschmidt's 
optimal range of uniformity in the arguments (inputs) of the logarithms.
(This feature was  first introduced in~\cite[\emph{parameter~$M$ on page~179}]{Waldschmidt1993}, see also~\cite[\emph{Theorem~9.1 (ii) with the modified parameter~$B$}]{WaldschmidtBook}.)

Once the reduction  to softer bounds on linear forms in two logarithms (given Waldschmidt's uniformity feature) had been made,
it opened up a second and more algebraic path to~\cref{main Diophantine theorem}  independent from the analytic  theory of linear forms in logarithms. 
This was carried out by Bombieri in the 1990s, grounded in a Galois equivariant form of the Thue--Siegel principle found by Bombieri, van der Poorten, and Vaaler~\cite{BombieriThueSurvey,BombieriVdPVaaler}.
A general and completely explicit bound~\cref{effective height bound}, which for large enough~$h(A)$ is also of the Baker--W\"ustholz form~\cref{depend C} (although with an inferior dependence on the implicit parameters~$[K:\Q], v, \varepsilon$, and $t = \rank(\Gamma)$), can be found in Bombieri~\cite{BombieriGm} in the Archimedean case, and Bombieri--Cohen~\cite{BombieriCohenII} and~\cite{BombieriCohenElements} in the nonarchimedean case. Bombieri's original work~\cite{BombieriGm} was
founded upon Viola's geometric form~\cite{ViolaDyson} of Dyson's nonvanishing lemma, whereupon the introduction of an extrapolation parameter  thereafter in~\cite{BombieriCohenII,BombieriCohenElements}  denoted ``$M$''
 led in~\cite{BombieriCohenII} (worked out in detail only in the nonarchimedean case)
to a logarithmically improved implicit coefficient in the final estimate of the form~\cref{depend C}. Finally in~\cite{BombieriCohenElements} the appeal to Viola's theorem was bypassed and replaced by an elementary Wronskian argument,  
and the final implicit coefficient in the form~\cref{depend C} (once again for $h(A) \gg_{[K:\Q],v,\varepsilon,t} 1$, and worked out there in the nonarchimedean case) was improved some further.

Ultimately, both the Baker--Waldschmidt and 
the Thue--Siegel (Bombieri) paths to~\cref{main Diophantine theorem}  
(the only known proofs)
remain  fairly complicated,  
and there is no easy path known to such a  theorem, even in the special case~$A=1$. 
This is despite
 the central role of~\cref{main Diophantine theorem}
in modern Diophantine analysis, and after several decades of intensive research (see \cite{BugeaudBook,Baker,BakerWustholzSurvey,FeldmanNesterenko} for a panorama of this subject, and to~\cite{Bombieri1999,BombieriFortyYears,CohenPerspectives} for an introduction to the paradigm's history and some speculation about its future in the light of the famous and elusive $abc$ conjecture).
In this paper, we give a new---and perhaps simpler---proof of~\cref{main Diophantine theorem} by applying our quantitative arithmetic holonomy bounds that we develop below.
Explicit bounds will be given in a future paper as we continue to explore the potential of the method we outline here. Two longstanding unsolved problems in the area are to give an effective height bound on the multivariable $S$-unit equation, and to effectivize Thue's theorem for general algebraic numbers besides those of the binomial form $\sqrt[r]{a/b}$.

\subsection{Notation.} \label{notation} 
As we already said in the introduction, all our conventions and notations for heights and absolute values are the ones found in Bombieri and Gubler's book~\cite{BombieriGubler}.  
We spell out, in particular, that $|\cdot|_v$ is normalized so that $|p|_v=p^{-\frac{[K_v:\Q_p]}{[K:\Q]}}$ if $v\mid p$, and $|x|_v=|x|_\infty^{\frac{[K_v:\R]}{[K:\Q]}}$ for $v\mid \infty$ and $|\cdot|_\infty$  the  Euclidean norm on $K_v\subset \C$. 
 We let $\D_v(R)$ and $\Db_v(R)$ denote the open and closed unit discs of radius~$R$ in $\C_v$, and~$\T_v(R) = \partial \Db_v(R) := \{|x|_v  = R\}$;
 we sometimes omit the subscript when~$v=\infty$, and we also abridge $\D_v := \D_v(1)$, $\Db_v := \Db_v(1)$, $\T_v := \T_v(1)$. 
More generally, we follow the notation of~\cite{L2chi}. We will use $c_0$, $c_1, \ldots$ to denote a sequence of effective constants readily computable 
in a constructive function of a datum indicated within (possibly empty) brackets. 
Similarly, the Vinogradov symbols~$\ll$, $\asymp$, and~$\gg$ used without subscripts will imply absolute effective constants, 
and in general implicit constants that are effectively and straightforwardly computable in terms of the data listed in the subscript. 
The semantical meaning of every occurrence of those symbols is that effective implicit conditions can be written down for which the given statement holds.

\section{Arithmetic holonomy bounds: a quantitative theory. 
}  \label{sec:quantitative}

\subsection{Holonomy bounds and irrationality.}\label{sec: irrational}
Let us recall the framework in~\cite[\S~1.2, \S~2]{L2chi} for proving the irrationality of certain
periods~$\eta$.
We start (typically) with $A(x)\in \Z\llbracket x \rrbracket$ and~$B(x) \in \Q\llbracket x \rrbracket$ such that $B(x)-\eta A(x)$  has larger convergence radius than both $A(x)$ and~$B(x)$ as power series in $\R\llbracket x \rrbracket$.
By analogy with~\cite{Apery,BeukersG}, we call~$\eta$ an Ap\'{e}ry limit.
We pick a holomorphic mapping $\varphi: (\Db,0) \rightarrow (\C, 0)$ such that $\varphi^*(B(x)-\eta A(x))$ is meromorphic on $\D \subset \Db$. We then  bound the largest number of
 $\Q(x)$-linearly independent power series $f_1, \dots, f_m \in \Q \llbracket x \rrbracket$ such that the denominators of their coefficients 
 have the same shape (or better) as those
 of $B(x)$, and additionally that $\varphi^* f_i$ is meromorphic on $\D$ for all $i$.
  If we assume for contradiction that $\eta \in \Q$, then
 we can take some of these $f_i$'s to be $B(x)-\eta A(x)$ and its derivatives, but there may exist
 (unconditionally) 
 additional such functions. Our holonomy bound (see \cref{hol_bound} below) 
  provides an upper bound of $m$ in terms of $\varphi$ and denominator types when $|\varphi'(0)|$ is sufficiently large comparing to the denominators. We deduce that $\eta \notin \Q$ precisely when we find more functions $f_i$ (including $B(x)-\eta A(x)$ and its derivatives) than the upper bound provided by the following theorem proved in~\cite[Theorem~2.5.1]{L2chi}. 
  
\begin{theorem}\label{hol_bound}
Consider two positive integers~$m,r \in \NwithoutzeroA$ and an $m \times r$ rectangular array of nonnegative real numbers 
$\mathbf{b} := \big( b_{i,j} \big)_{\substack{ 1 \leq i \leq m, \,
1 \leq j \leq r }},$
 all of whose columns have the form
$ 0 = b_{1,j} =\cdots = b_{u_j,j} < b_{u_j+1,j}= \cdots = b_{m,j}=: b_j$, $\forall j = 1, \ldots, r,$
for some $u_j \in \{0, 1, \ldots, m\}$ depending on the column. 
Let
$
\sigma_i := b_{i,1} +\ldots + b_{i,r}, \, i = 1, \ldots, m
$
be the $i$-th row sum, and define
$\displaystyle
\tau(\mathbf{b}) := \frac{1}{m^2} \sum_{i=1}^{m} (2i-1) \sigma_i = \sigma_m -\frac{1}{m^2}\sum_{j=1}^r u_j^2 b_j \in [0, \sigma_m]$.

Consider a holomorphic mapping $\varphi : (\Db, 0) \to (\C,0)$ as above with derivative ({\it conformal size}) satisfying the condition $ \log{|\varphi'(0)|} > \tau(\bb)$. 
Suppose there exists an $m$-tuple $f_1, \ldots, f_m \in \Q \llbracket x \rrbracket$ of $\Q(x)$-linearly independent formal functions with denominator types of the form
\begin{equation}\label{den type 1}
f_i(x) = \sum_{n=0}^{\infty} a_{i,n} \frac{x^n}{ [ 1, \ldots, b_{i,1} \cdot n] \cdots [1,\ldots, b_{i,r} \cdot n]}, \qquad a_{i,n} \in \Z, 
\end{equation}
 such that $f_i(\varphi(z)) \in \C \llbracket z \rrbracket$ is the germ of a meromorphic function on $|z| < 1$ for all $i = 1, \ldots, m$; assume that all $f_i$ are holonomic.
 Then 
\begin{equation}\label{BCbound} 
m  \leq  \frac{ \iint_{\T^2} \log|\varphi(z)-\varphi(w)| \, \mv(z) \mv(w) }{  \log{|\varphi'(0)|} - \tau(\mathbf{b}) }. 
\end{equation}
\end{theorem}

Let $\vec{b}:=(b_{1},\dots, b_{r})$ and $\cV(\varphi, \vec{b})$ denote the $\Q(x)$-span of power series $f$ where $\varphi^* f$ is meromorphic on $\D$ and  
\begin{equation}\label{den type}
f(x) = \sum_{n=0}^{\infty} a_{n} \frac{x^n}{ [ 1, \ldots, b_{1} \cdot n] \cdots [1,\ldots, b_{r} \cdot n]}, \qquad a_{n} \in \Z.
\end{equation}
When $\log |\varphi'(0)| > \sum_{j=1}^r b_{j}$, \cref{hol_bound} provides an upper bound of $\dim_{\Q(x)} \cV(\varphi, \vec{b})$. 

\begin{remark}
In \cite[\S\S~6--8]{L2chi}, we allow more general denominators of $f_i$, and we have various improvements and alternative bounds. For the rest of this section, we will discuss upper bounds on the irrationality measure of $\eta$ and use the bound \cref{BCbound} as a showcase. We hope it will be clear to the reader how the proof idea of \cref{theotrue} can be used to refine all our qualitative holonomy bounds in~\cite[\S\S~6--8]{L2chi} to include an irrationality measure. By the same token, simpler proofs of the bounds needed for the application to~\cref{main Diophantine theorem} 
can be more directly obtained with the Perelli--Zannier dynamic box principle technique that we explain in~\cite[Appendix~B]{L2chi}.
\end{remark}

\subsection{Irrationality measures.}
We recall the definition of irrationality measure. 
\begin{definition}  \label{def: irrationality measure} Let~$v$ be a place of~$\Q$.
The irrationality exponent/measure of $\eta \in \Q_v \setminus \Q$, denoted by $\mu(\eta)$, is the supremum of the set of numbers $\kappa$ such that 
$\displaystyle
0< \left|\eta - \frac{p}{q} \right|_v < \frac{1}{q^\kappa}
$
is satisfied by infinitely many $(p,q)\in \Z \times \Z_{>0}$.
\end{definition}
In other words, for $\kappa > \mu(\eta)$, the above inequality only has finitely many solutions $(p,q)$ as above; thus there exists a constant $C=C(\eta, \kappa)\in \R_{>0}$ such that for all $(p,q)\in \Z \times \Z_{>0}$,
we have
$\displaystyle \left|\eta - \frac{p}{q} \right|_v > \frac{C}{q^\kappa}$.
When~$\eta$ is algebraic of degree~$d \geq 2$, Liouville's theorem~\cite[ch.~1 \S~1.6]{FeldmanNesterenko} gives $\kappa = d$ with an 
effective constant~$C$, and in this case there is a distinction to be made between~\cref{def: irrationality measure} (which, by Roth's theorem, reads $\mu(\eta) = 2$ whenever $2 \leq [\Q(\eta):\Q] < \infty$), and exponents~$\kappa$ for which an explicit $C(\eta,\kappa)$ can be provided. 
The infimum of the latter~$\kappa$, though this not a mathematical definition, is commonly denoted~$\mueff(\eta)$, and it is the notion featured in all our theorems in this paper (see \cref{rmk: effC}). We denote by $\mueffKv(\eta)$ the straightforward generalization over a number field~$K$ and place $v \in M_K$ for $v$-adic $K$-rational approximations to an $\eta \in K_v \setminus K$.

\cref{hol_bound} extends to a quantitative refinement that includes an explicit upper bound on the irrationality measure~$\mu(\eta)$ 
for an
 Ap\'ery limit $\eta \in \R\setminus \Q$ as considered in~\cref{sec: irrational};
for concreteness we formulate our next proposition for the case of the Archimedean place of~$\Q$.
Let $\cV(\varphi, \vec{b}, \eta)$ denote the $\R(x)$-span of power series $g\in \R \llbracket x \rrbracket$ of form $g(x)=B(x) - \eta A(x)$
such that $\varphi^*g$ is meromorphic on $\D$ and $A(x), B(x)$ satisfy \cref{den type}.
\begin{proposition} \label{theomainreprise}
Fix
$\varphi : (\Db,0) \to (\C,0)$ and a denominator type~$\vec{b}$
with
$|\varphi'(0)| > e^{\sum_{j=1}^r b_j}$.   Let
$$m = \dim_{\R(x)}\cV(\varphi, \vec{b}, \eta) \ge 
 \dim_{\Q(x)}\cV(\varphi, \vec{b})=\gamma \cdot m, \quad \gamma \in \Q \cap (0,1].$$ 
Take $\rho \in (0,1]$ so that there exists a basis $\{g_i=B_i - \eta A_i\}$ of $\cV(\varphi, \vec{b}, \eta)$ where~$\varphi^*A_i$ and~$\varphi^*B_i$ both individually converge on~$|z| < \rho$.
Consider $\bb$ as in \cref{hol_bound} such that $A_i, B_i$ satisfy \cref{den type 1}. Assume that all $g_i$ are holonomic. 
Then
\begin{equation} \label{holokappaigen}
m \leq \frac{\iint_{\T^2} \log{|\varphi(z)-\varphi(w)|} \, \mv(z)\mv(w)   }{\log{|\varphi'(0)| }- \tau(\bb)  - (1-\gamma) \left( 2 \mu(\eta)^{-1} - (1-\gamma)\mu(\eta)^{-2} \right)  \log{   \frac{1}{ \rho }   }},
\end{equation}
conditionally on the positive denominator. 
\end{proposition}
If we set $\mu(\eta)\rightarrow \infty$ in \cref{holokappaigen}, we recover \cref{BCbound}. In particular, every qualitative irrationality proof by the holonomy bounds method of~\cref{hol_bound} automatically refines via~\cref{holokappaigen} to an explicit measure of irrationality. 
It must  be said, however, that in this form of the result,  the extra coefficient~$2$ of~$\mu(\eta)^{-1}$ in~\cref{holokappaigen} 
means that, when the \emph{usual} Ap\'{e}ry limit works, \cref{theomainreprise} does \emph{not}
recover the most basic standard bound on the irrationality measure.
It remains an open problem to give an $\varepsilon$-improvement over the usual irrationality measure in every single case. 

\begin{remark} \label{withintegrationsremark}
Consider the general denominator type as in \cite[Thm~6.0.2 and~7.0.1]{L2chi}, i.e., $A_i, B_i$ are of the form 
\[a_{i,0} +  \sum_{n=1}^{\infty} a_{i,n} \frac{x^n}{n^{e_i} [ 1, \ldots, b_{i,1} \cdot n] \cdots [1,\ldots, b_{i,r} \cdot n]}, \qquad a_{i,n} \in \Z,\] where $\be=(e_1,\dots, e_m)\in \N^m$. 
Define~$\tau^{\sharp}(\mathbf{e})$ as in~\cite[\S 6]{L2chi}.
Then we have the following bound on 
 $\mu(\eta)$:
\begin{equation}  \label{withintegrations}
m \leq \frac{\iint_{\T^2} \log{|\varphi(z)-\varphi(w)|} \, \mv(z)\mv(w)   }{\log{|\varphi'(0)| }- \tau(\bb) - \tau^\sharp(\be) - (1-\gamma) \left( 2 \mu(\eta)^{-1} - (1-\gamma)\mu(\eta)^{-2} \right)  \log{   \frac{1}{ \rho }   }}.
\end{equation}
\end{remark}

\Cref{theomainreprise} is a special case, with $K=\Q$ and $S=\{v\}$ a singleton, of \Cref{theotrue} below. For our more general setup, we consider
a number field
 $K$ together with a finite subset $S\subset M_K$ of its places, whereby $K$-rational approximations (sometimes restricted to be from a subfield of~$K$) are to be made simultaneously in all places from~$S$. 
We organize the targets of the approximation into a ``soft'' collection $\eta=(\eta_u)_{u\in S}$ 
of local elements $\eta_u \in K_u$, which in theory do not have to be related, but in practice always have some  sort of structural connection, although not always by a rule as simple 
as to take the $u$-embedding of an underlying element~$\eta \in K$.
One central example of this, which in a Kummer equivariant form occurs within the proofs of \cref{effective Thue-Siegel-Mahler,effective p-adic}, is similar to the equivariant setup~\cite[page~75]{BombieriThueSurvey} of Bombieri, van der Poorten, and Vaaler: 
the approximants~$\beta$ are to be drawn from a subfield~$F \subset K$, with $K/F$ being Galois and the set~$S \subset M_K$ being~$\Gal(K/F)$-stable, and
the assignment $u \rightsquigarrow \eta_u$ transforms according to a given projective representation $\pi : \Gal(K/F) \to \PGL_2(F)$.

Following the framework of Andr\'e~\cite[\S~VIII]{AndreG} but now combined with axiomatics of Ap\'ery limits, we consider a collection $(\varphi_v)_{v\in M_K}$
of analytic functions on an unspecified open neighborhood of~$\Db_v$, with  $\varphi_v(z)=z$ for all but finitely many~$v$; and a collection $(\rho_u)_{u\in S}$
of radii~$\rho_u \in (0,1]$. 
These data are to be used for imposing and  controlling convergence properties of the formal functions in the Ap\'ery limits method. 
Though one may consider more general linear independence settings, the focus here on single Ap\'ery limits, but used in a power system $\{1,H,\ldots,H^k\}$,
compels us to  consider only the powers $\eta^j := (\eta_u^j)_{u \in S}$ of the original Ap\'ery limit, along with the implied $K$-rational approximations $\beta^j$. 
We 
define $\cV(\{\varphi_v\}, \{\rho_u\}, \vec{b}, \eta, \nu)$ to be the formal\footnote{This we take to mean to consider in every realization $u \in S$ the $K_u( x)$-linear span of the power series $\sum_{j=0}^\nu \eta_u^j A_j(x)
\in K_u\llbracket x \rrbracket$. A $K(\eta,x)$-linear dependence is defined to be a $K(\eta,x)$-linear form which evaluates to a $K_u(x)$-linear dependence simultaneously in 
all places~$u \in S$.} $K(\eta, x)$-linear span of all formal power series $g\in K(\eta) \llbracket x \rrbracket$ of the form $g(x)=\sum_{j=0}^\nu \eta^j A_j(x) $ such 
that:
\begin{enumerate}
\item
For all $u\in S$, $\varphi_u^*g \in \hM(\D_u)$ is meromorphic on $\D_u$,  and every individual $\varphi_u^*A_j$ is convergent on $|z|_u < \rho_u$. 
\item For all $v\in M_K \setminus S$,  every individual $\varphi_v^*A_j$ is meromorphic on $\D_v$. 
\item There exists a finite subset $S' \subset M_K$ for which $A_0, \ldots, A_{\nu} \in \sum_{n=0}^{\infty} \frac{x^n \,  \cO_{K,S'}}{ [ 1, \ldots, b_{1} \cdot n] \cdots [1,\ldots, b_{r} \cdot n]}$. 
\end{enumerate}

\begin{theorem} \label{theotrue}
Given $(\varphi_v)_{v\in M_K}, \eta=(\eta_u)_{u\in S}, (\rho_u)_{u\in S} \in (0,1]^S$ as above, consider $r, \mu \in \N$ and, for $\nu = 0,1, \ldots, \mu$, a vector $\vec{b}_{\nu} \in [0,\infty)^{r}$, and a list of \emph{holonomic} (over~$K(x)$) elements
\[g_{\nu, 1}, \ldots, g_{\nu, m_{\nu}} \in \cV(\{\varphi_v\}, \{\rho_u\}, \vec{b}_{\nu}, \eta, \nu) \setminus \cV(\{\varphi_v\}, \{\rho_u\}, \vec{b}_{\nu}, \eta, \nu-1)\] satisfying conditions (1)--(3) above. Assume the totality $\{g_{\nu, i}\}_{0\leq \nu \leq \mu, 1\leq i \leq m_\nu}$ of these~$m = m_0+ \ldots + m_{\mu}$ holonomic functions to be $K(\eta,x)$-linearly independent, and
order them in some arbitrary way. Consider for the chosen ordering a denominators capping $m \times r$-matrix~$\bb$ obeying the constraint spelled out in \cref{hol_bound}, 
and having its $k$-th row vector dominate entrywise the~$\vec{b}$-vector $\vec{b}_{\nu}$ of the $k$-th function $g_{\nu,i}$. 

Consider a vector of exponents $\boldsymbol{\kappa} = (\kappa_u)_{u \in S} \in \R_{>0}^{S}$ subjected to the 
inequalities: 
\begin{equation}  \label{here it is}
\begin{aligned}
m  & >  \frac{ \sum_{v \in M_K, v\mid \infty} \iint_{\T_v^2} \log{|\varphi_v(z)-\varphi_v(w)|} \, \mv + \sum_{v \in M_K, v\nmid \infty} \sup_{z\in \T_v} \log |\varphi_v(z)|_v}{ \sum_{v \in M_K} \log{|\varphi'(0)|_v}
- \tau(\bb) 
-  \sum_{u\in S} \log{\rho_u^{-1}} +  \frac{ \left( \sum_{u\in S} \kappa_u -  \sum_{\nu=0}^{\mu}  \nu m_{\nu}/m  \right)^2}{ \sum_{u\in S} \frac{\kappa_u^2}{\log{\rho_u^{-1}}} }},  \\
\kappa_u & \leq   \frac{   \log{\rho_u^{-1}}  \sum_{v \in S} \frac{\kappa_v^2}{\log{\rho_v^{-1}}}  }{  \sum_{v \in S} \kappa_v -  \sum_{\nu=0}^{\mu}  \nu m_{\nu}/m  },  \quad \forall u \in S.
 \end{aligned}
\end{equation}
Then there exists a constant~$C \in \R_{> 0}$ such that the following Diophantine property is in place: 
\begin{equation} \label{star}
\text{For all~$\beta \in K$ with height $h(\beta) > C$, there exists a place~$v \in S$ 
where}
\ 
|\eta_v - \beta|_v \geq  H(\beta)^{-\kappa_v}. 
\end{equation}
 \end{theorem}

\begin{remark}\label{rmk: effC} 
By Northcott's theorem~\cite[Theorem~1.6.8]{BombieriGubler}, \cref{star} is another way to say that all but 
finitely many~$\beta \in K$ have some $v = v(\beta)$ where $|\eta_v - \beta|_v \geq  H(\beta)^{-\kappa_v}$. As in~\cref{main Diophantine theorem}, the point of writing it this 
way
is that~$C$ in our proof can readily be made explicit in terms of meromorphic presentations $\varphi_u^* g_{\nu,i} = R^u_{\nu,i}/V^u_{\nu,i}$ and $\varphi_v^* A_{\nu,i,j} = S^v_{\nu,i,j}/W^v_{\nu,i,j}$ with $V^v_{\nu,i}(0) = W^v_{\nu,i,j}(0) = 1$ along with the data 
 $\big( |\eta_u|_u, \kappa_u, \rho_u \big)_{u \in S}$, $\bb$, 
  $\sup_{\T_v}{\log{ | R^u_{\nu,i}|}}$, $\sup_{\T_v}{\log{ | V^u_{\nu,i}|}}$, $\sup_{\T_v(\rho_v)} \log{ |S^u_{\nu,i,j}|}$, and $\sup_{\T_v(\rho_v)} \log{ |W^u_{\nu,i,j}|}$
 for all $u \in S$ and $v \in M_K$ (extending $\rho_v := 1$ to $v \in M_K \setminus S$), in a linear dependence on those four suprema uniformly in terms of the Chudnovsky--Osgood functional bad approximability constant for the system  $\{ A_{\nu,i,j} \}$, as discussed in~\cite[\S~3.2]{L2chi} and bounded explicitly by Chudnovsky's~\cite[Theorem 3.2.10]{L2chi} in the case of a differentially closed system. 
 \end{remark}

\begin{remark}
Comparing to~\cref{BCbound}, the extra term subtracted off in the denominator of~\cref{here it is} is 
$
  \sum_{u\in S} \alpha_u -  \left( \sum_{u\in S} \kappa_u - E\right)^2 / \left(\sum_{u \in S} \kappa_u^2/\alpha_u \right),
$
where 
$
\alpha_u := \log{\rho_u^{-1}}$, $\gamma_{\nu} := m_{\nu}/m \textrm{ with } \sum_{\nu=0}^{\mu} \gamma_{\nu} = 1$, $ E := 
 \sum_{\nu=0}^{\mu}  \nu \gamma_{\nu} \in [0,\mu]$.
The positivity of this term for the case~$\sum_{u\in S} \kappa_u > E$ follows by the Cauchy--Schwarz inequality.
In the proportionality regime $\kappa_u := \alpha_u T$, the second condition in~\cref{here it is} is automatically fulfilled, and the subtracted term reduces to  
$2/(ET) -   (E/T)^2 /\left( 
\sum_{u} \alpha_u \right)$.
This with $T \to \infty$ is how \cref{theotrue} subsumes \cref{hol_bound}.
\end{remark}

\subsection{Proof of \cref{theotrue}.}
We use the framework of Bost's slopes method~\cite[\S\S 4.1, 4.2]{BostFoliations}. 
Without loss of generality upon permuting the columns of~$\bb$, the indices~$u_i$ in the condition
spelled out in~\cref{hol_bound} satisfy $u_1 \leq u_2 \leq \ldots \leq u_r$. 
In our ordering, let us denote by~$g_i$ the $i$-th function $\{g_{\nu, j}\}$, and by~$\nu(i)$ the index for which $g_i \in \cV(\{\varphi_v\}, \{\rho_u\}, \vec{b}, \eta, \nu(i)) \setminus \cV(\{\varphi_v\}, \{\rho_u\}, \vec{b}, \eta, \nu(i)-1)$.

We present the argument in its qualitative form: assume for contradiction that the constant $C$ does not exist. 
This means there is some $\varepsilon>0$ 
and an infinite sequence of $\beta \in K$ with $H(\beta)\rightarrow \infty$ 
along which
$
|\eta_u - \beta|_u < H(\beta)^{-\kappa_u + \varepsilon}, 
$
simultaneously for all~$u \in S$. 
(We shall let $\varepsilon \rightarrow 0$ at the end of the proof.)
Consider the following $\cO_K$-module of rank $m(D+1)$: 
\[
E_D:= \bigoplus_{h=0}^r \bigoplus_{i=u_h+1}^{u_{h+1}} \frac{1}{[ 1, \ldots, u_{h+1}b_{h+1} D] \cdots [1,\ldots, u_rb_r D]} \, y_i \cdot \cO_K[1/x]_{\leq D},
\]
where the orthogonal direct sum is in the category of Hermitian vector bundles over $\Spec \cO_K$:
for $v\nmid \infty$,
the vectors $y_i, y_i x^{-1}, \dots, y_i x^{-D}$ are equipped with the norm
$\| y_i x^k \|_v := \max(| \beta |_v,1)^{\nu(i)} (\max_{ \T_v} \log |\varphi_v|_v)^k$.
For $v\mid \infty$, the Hermitian metric is defined
 (as in~\cite[\S\S~7.3.1, 7.3.3]{L2chi})
 for $s \in \cO_K[1/x]_{\leq D}$
 by $\| y_i s \|_v:= \max(| \beta |_v,1)^{\nu(i)} \| s\|_{\mathrm{BC}}$, where $\| \cdot \|_{\mathrm{BC}}$ signifies the Bost--Charles metric~\cite[\S\S~8.2-8.3]{BostCharles}. We use $\ovE_D$ to denote $E_D$ 
 with this Hermitian metric.
Approximately speaking, we will choose an optimizing $D$ in linear proportion with $h(\beta)=\log H(\beta)$. This allows us to 
 focus on $D\rightarrow \infty$ asymptotic estimates as we will let $h(\beta)\rightarrow \infty$  running along our sequence of counterexamples~$\beta \in K$.
 It will be apparent from the local estimates that the argument is actually effective. 
 
 By the arithmetic Hilbert--Samuel theorem and a theorem of Bost and Charles \cite[Theorem~5.4.1 and Proposition~5.4.2]{BostCharles}, we have
\begin{equation*}
\begin{aligned}
\ardeg \ovE_D =  & \left(\frac{m}{2}\left(\sum_{v\mid \infty} \iint_{\T_v^2} \log{|\varphi_v(z)-\varphi_v(w)|} \, \mv + \sum_{v\nmid \infty} \sup_{\T_v} \log |\varphi_v|_v \right) +\sum_{h=1}^r u_h^2 b_h  \right)D^{2} \\
& -\left( \sum_{\nu=0}^{\mu}  \nu \gamma_{\nu} \right) h(\beta)  \, mD +o(D^{2}+Dh(\beta)).
\end{aligned}
\end{equation*}

Write $g_i(x)= \sum_{j=0}^\mu \eta^j A_{i,j}(x)$. We evaluate~$E_D$
on 
$
y_i :=  \sum_{j=0}^\mu \beta^j A_{i,j}(x),
$
and thus for $h(\beta) \gg 1$ we obtain an injective homomorphism 
$
\psi_D : E_D \hookrightarrow x^{-D} K\llbracket x \rrbracket. 
$
We filter $x^{-D} K\llbracket x \rrbracket$ as
$x^{-D} K\llbracket x \rrbracket\supseteq x^{1-D} K\llbracket x \rrbracket \supseteq \ldots \supseteq x^{n-D} K\llbracket x \rrbracket\supseteq \ldots$,
and we metricize the graded quotients $x^{n-D} K\llbracket x \rrbracket/x^{n+1-D} K\llbracket x \rrbracket$ by $\| x^{n-D} \|_v=1$ for all $v\in M_K$. 
We define ${E}^{(n)}_{D} := \psi_D^{-1} \left( x^{n-D} K\llbracket x \rrbracket\right) $. For each $n \in \N$, the evaluation map $\psi_D$ induces an injective homomorphism 
$\psi_D^{(n)}: {E}^{(n)}_{D}/ {E}^{(n+1)}_{D}\hookrightarrow x^{n-D} K\llbracket x \rrbracket/x^{n+1-D} K\llbracket x \rrbracket.$
Therefore, $\rank \left( {E}^{(n)}_{D}/{E}^{(n+1)}_{D} \right) \in \{0,1\}$. Let 
$\mathcal{V}_{D} := \left\{n\in \N \, : \, \rank \left( {E}^{(n)}_{D}/{E}^{(n+1)}_{D} \right)=1 \right\}$. We have $\# \mathcal{V}_{D} = \rank E_D=m(D+1)$.
By the Chudnovsky--Osgood functional bad approximability theorem~\cite[Theorem~3.2.13]{L2chi}, we have  $\cV_D \subset \{0, 1, \ldots, (m+\delta)D\}$ once $\delta>0$ and $D\gg_{\delta} 1$. 

Bost's slopes inequality in our setting gives $\ardeg \ovE_D \leq \sum_{n\in \cV_D} \sum_{v\in M_K} h_v(\psi_D^{(n)})$.  We provide an upper bound on each local evaluation height $h_v(\psi_D^{(n)})$ as follows. For the ease of notation, we will adopt the convention that $\frac{0^2}{\log{1^{-1}}} = 0$ and set $\rho_v:=1$, $\kappa_v:=0$, and $\eta_v:=\beta$ for $v\notin S$. Write $h_i(x):=f_i(x) - g_i(x) = \sum_{j=0}^\mu (\beta^j - \eta^j) A_{i,j}(x)$. Let $\cL$ denote the line bundle $\cO([0])$ on $\bP^1_{\cO_K}$.

\begin{lemma}
For $v\mid \infty$, we have (here we can replace $\varepsilon$ by $0$ if $v\notin S$) the local  evaluation height upper bound: 
\[h_v(\psi_D^{(n)})  \leq -n \log |\varphi'_v(0)|_v+ D\iint_{\T_v^2} \log{|\varphi_v(z)-\varphi_v(w)|} \, \mv + (-(\kappa_v-\varepsilon) h(\beta) + n \log (1/\rho_v))^+ + o(n).\]
\end{lemma}

\begin{proof}
For simplicity of notation, we drop $v$ from the subscript and set $|\beta|_+:=\max(|\beta|,1)$. 
An easy reduction, see~\cite[Lemma~7.4.1]{L2chi}, lets us to assume that all $\varphi^* g_i \in \hO(\Db)$, as well as all $\varphi^* A_{i,j} \in \hO(\Db_{\rho})$,  are holomorphic  on their respective closed discs, rather than merely meromorphic on the open discs.  

Given $\{P_i\}_{i=1}^m \subset K[1/x]_{\leq D}$ with $\|y_i P_i\| \leq 1$, set $s:=\sum_{i=1}^m P_i f_i = c_n x^{n-D} + \ldots$, which is a section of the line bundle $\cL^{\otimes D}$ on the formal neighborhood of $0$ in $\bP^1_K$; we write $s_1:=\sum_{i=1}^m P_i g_i$ and $s_2:=\sum_{i=1}^m P_i h_i$.
Let $(h_i(\varphi(z))_{\leq n}$ denote the $O(z^{n+1})$ truncation of the $z$-power series $h_i(\varphi(z))$.
Note that $x^D \cdot s$ is a formal function on $\bP^1_K$ and $\varphi^*(x^D \cdot s) = c_n \varphi'(0)^n z^n + \ldots$. The function 
$F(z):= \varphi^*(x^D \cdot s_1) + \sum_{i=1}^m \varphi^*(x^D P_i) \cdot (h_i(\varphi(z))_{\leq n}$ is holomorphic on $\Db$ and has the same lowest order term as $\varphi^*(x^D \cdot s) $.
Therefore $\log |z^{-n} F(z)|$ is subharmonic, and
\begin{equation*}
\begin{aligned}
& \log |c_n| + n \log |\varphi'(0)| \leq \int_\T \log |F(z)| \mv \\
& \leq \int_\T \log \sup_{1\leq i \leq m} \{ |\varphi^*(x^D P_i) \cdot \varphi^*g_i|, |\varphi^*(x^D P_i) \cdot  (h_i(\varphi(z))_{\leq n}|\} \mv  + O(1)\\
& \leq  \int_\T \log \sup_{1\leq i \leq m} \{ |\varphi^*(x^D P_i)|\cdot |\beta|_+^{\nu(i)}\} \mv  + (-\kappa h(\beta) + n \log (1/\rho))^+ + O(1)\\
& = (-\kappa h(\beta) + n \log (1/\rho))^+ - \int_\T \log \| \varphi^* x^{-D}\|_{\varphi^* \ovcL^{\otimes D}} \mv
 + \int_\T \log \sup_{1\leq i \leq m} \| \varphi^* P_i\|_{\varphi^* \ovcL^{\otimes D}} \cdot |\beta|_+^{\nu(i)} \mv + O(1) \\
&= (-\kappa h(\beta) + n \log (1/\rho))^+ + D \iint_{\T^2} \log{|\varphi_v(z)-\varphi_v(w)|} \, \mv 
  + \int_\T \log \sup_{1\leq i \leq m} \| \varphi^* P_i\|_{\varphi^* \ovcL^{\otimes D}} \cdot |\beta|_+^{\nu(i)} \mv + O(1).
\end{aligned}
\end{equation*}
For the third inequality, note that on $\T$, 
\[|\varphi^*g_i|\cdot |\beta|_+^{-\nu(i)} = O(1), \quad | (h_i(\varphi(z))_{\leq n}|\cdot |\beta|_+^{-\nu(i)}=O\big(|\beta - \eta| \max_{0\leq l \leq n, 1\leq i \leq m, 0\leq j \leq \mu} |b_{i,j,l}| \big)=O(H(\beta)^{-\kappa+\varepsilon} \rho^{-n}),\]
where $b_{i,j,l}$ is the $z^l$-coefficient of $\varphi^* A_{i,j}$.
At the fourth and the fifth lines, by a mild abuse of notation upon embedding by the given $v: K \hookrightarrow \C$, we use $\ovcL$ to denote the line bundle $\cL_\C$ equipped with the Bost--Charles Hermitian metric as in \cite[\S~7.3.1]{L2chi} (i.e., we ignore all other places, only focus on $v: K \hookrightarrow \C$). The final equality is a computation due to Bost and Charles performed in~\cite[(7.3.5), (7.4.3)]{L2chi}.

It remains to bound $\displaystyle \int_\T \log \sup_{1\leq i \leq m} \| \varphi^* P_i\|_{\varphi^* \ovcL^{\otimes D}}\cdot |\beta|_+^{\nu(i)} \mv$ under the condition $\log \sup_{1\leq i \leq m} \|y_i P_i \| \leq 0$: 
\[
\begin{aligned}
& \int_\T \log \sup_{1\leq i \leq m} \| \varphi^* P_i\|_{\varphi^* \ovcL^{\otimes D}} \cdot |\beta|_+^{\nu(i)} \mv \leq \log \int_\T \sup_{1\leq i \leq m} \| \varphi^* P_i\|_{\varphi^* \ovcL^{\otimes D}} \cdot |\beta|_+^{\nu(i)} \mv \\
& \leq \log \int_\T \sum_{i=1}^m \| \varphi^* P_i\|_{\varphi^* \ovcL^{\otimes D}} \cdot |\beta|_+^{\nu(i)} \mv  = \log \sum_{i=1}^m \int_\T  \| \varphi^* P_i\|_{\varphi^* \ovcL^{\otimes D}} \cdot |\beta|_+^{\nu(i)} \mv \\
& \leq \log \sup_{1\leq i \leq m} \| y_i P_i \| + \log m + o(n) \leq o(n).  
\end{aligned}
\]
\end{proof}
The estimate for $v\nmid \infty$ is similar, with the subharmonicity of $\log |z^{-n} F(z)|$ being replaced by the maximum principle for $z^{-n} F(z)$.
Summing the local estimates we derive for the $n$-th global evaluation height: 
\[
\begin{aligned}
\sum_{v\in M_K^{\rm fin}} h_v(\psi_D^{(n)})  = & \sum_{h=1}^rb_h \max\{n, u_h D\} - n \, \sum_{v \in M_K^{\rm fin}}   \log{|\varphi_v'(0)|_v}  +  D \, \sum_{v \in M_K^{\rm fin}}  \sup_{\T_v} \log |\varphi_v|_v \\
& + \sum_{u \in S}  \left( n \, \log(1/\rho_u) -  (\kappa_u -\varepsilon) h(\beta) \right)^+ +o(n).
\end{aligned}
\]
We plug our estimates of $\ardeg \ovE_D$ and $h_v(\psi_D^{(n)})$ into the slopes inequality $\ardeg \ovE_D \leq \sum_{n\in \cV_D} \sum_{v\in M_K} h_v(\psi_D^{(n)})$. The optimal choice of $D$ for a given $h(\beta)$ is found by minimizing
\begin{equation}\label{chi term}
\left( \sum_{\nu=0}^{\mu}  \nu \gamma_{\nu} \right)  \, m h(\beta)/D + \sum_{u\in S}  \int_0^m \max \{0, t \, \log{\rho_u^{-1}} - (\kappa_u -\varepsilon) h(\beta)/D \} \, dt.
\end{equation}
Set $q:=h(\beta)/D$, $\chi_u := \frac{ (\kappa_u - \varepsilon) q}{ \log{\rho_u^{-1}}}$ for $\rho_u <1$, and $\chi_u:=0$ otherwise. Later we will pick $q$ so that $\chi_u \leq m$ for all $u\in S$. Under this condition $\chi_u \leq m$, a direct computation transforms~\cref{chi term} into
\begin{equation*}
\begin{aligned}
& \left( \sum_{\nu=0}^{\mu}  \nu \gamma_{\nu} \right)  \, mq + \sum_{u\in S}  \left( \frac{m^2-\chi_u^2}{2} \log{\rho_u^{-1}} - (\kappa_u -\varepsilon) (m-\chi_u) q \right) \\
=   &
\left( \sum_{u\in S} \frac{(\kappa_u-\varepsilon)^2}{\log{\rho_u^{-1}}} \right) \, \frac{q^2}{2}   
- \left( \sum_{u\in S} (\kappa_u -\varepsilon) -  \sum_{\nu=0}^{\mu}  \nu \gamma_{\nu}    \right) \, mq
 + \left( \sum_{u\in S} \log{\rho_u^{-1}} \right) \,\frac{m^2}{2},
\end{aligned}
\end{equation*}
This is a quadratic form in~$q$, with minimum taken at 
$\displaystyle
 q := \frac{ \sum_{u\in S} (\kappa_u-\varepsilon) -  \sum_{\nu=0}^{\mu}  \nu \gamma_{\nu} }{ \sum_{u\in S} \frac{(\kappa_v-\varepsilon)^2}{\log{\rho_u^{-1}}} }m.
$
 Note that the condition $\chi_u <m$ does indeed follow from \cref{here it is} for $\varepsilon$ sufficiently small.
 
 Now we let $h(\beta)\rightarrow \infty$ and take $D:= \lfloor qh(\beta) \rfloor$. The leading (order $D^2$) terms in the slopes inequality give
 \[
 \begin{aligned}
&  \frac{m^2}{2} \left(\sum_{v \in M_K}   \log{|\varphi_v'(0)|_v}  -\tau(\bb) - \sum_{u\in S} \log{\rho_u^{-1}} +  \frac{ \left( \sum_{u\in S} (\kappa_u - \varepsilon) -  \sum_{\nu=0}^{\mu}  \nu \gamma_{\nu}   \right)^2}{ \sum_{u\in S} \frac{(\kappa_u-\varepsilon)^2}{\log{\rho_u^{-1}}} } 
\right) \\
\leq & \frac{m}{2}\left(\sum_{v\mid \infty} \iint_{\T_v^2} \log{|\varphi_v(z)-\varphi_v(w)|} \, \mv + \sum_{v\nmid \infty} \sup_{z\in \T_v} \log |\varphi_v|_v \right),
\end{aligned}
 \]
which with $\varepsilon \rightarrow 0$ converges to the thesis of the theorem.

 \section{A dihedral method.} \label{sec:Thue continue}
 We begin by recalling  the classical hypergeometric method~\cite[ch.~1, \S~3]{FeldmanNesterenko}, by which Thue originally proved his ineffective finiteness theorem on the integer solutions
 $(X,Y) \in \Z^2$ of the equation $aX^r - b Y^r = c$ for $r \geq 3$ and $a,b,c \in \Z \setminus \{0\}$.
 The hypergeometric equation with three parameters $\alpha, \beta, \gamma$ is the second-order linear ODE
$$
\begin{aligned}
x(x-1)\frac{d^2F}{dx^2} + \left( (\alpha+\beta + 1) x - \gamma \right) \frac{dF}{dx} + \alpha\beta F = 0,
 \end{aligned}
 $$
with one solution, regular at the singular point~$x = 0$, being given by the power series
$$
\begin{aligned}
 F(x) =  \pFq{2}{1}{\alpha, \beta}{\gamma}{x}  := \sum_{k=0}^{\infty} \frac{\alpha(\alpha+1) \cdots (\alpha + k - 1) \cdot \beta(\beta+1) \cdots (\beta+k-1)}{
 \gamma(\gamma+1) \cdots (\gamma+k-1)} \, \frac{x^k}{k!}.
 \end{aligned}
$$
A well-known fact going back to Gauss and Jacobi~\cite[page 160]{Jacobi} is that the Hermite--Pad\'e approximants to the hypergeometric functions of the special form $\pFq{2}{1}{\alpha, 1}{\gamma}{x}$
are all explicitly given in terms of hypergeometric polynomials themselves. This applies in particular to the logarithm function $\log(1-x) = - x \cdot \pFq{2}{1}{1, 1}{2}{x}$ as well as the binomial functions $(1-x)^{\nu} = \pFq{2}{1}{-\nu, 1}{1}{x}$, and we have: 
\begin{equation} 
\begin{aligned}
\label{hyper poly}
  & \quad \pFq{2}{1}{-\nu-n,-m}{-m-n}{x}  -   (1-x)^{\nu} \cdot  \pFq{2}{1}{\nu-m,,-n}{-m-n}{x}   
 \\ 
 &  = (-1)^m \frac{\binom{n+\nu}{m+n+1}}{\binom{m+n}{m}}    \pFq{2}{1}{-\nu + m+1 ,n+1}{m+n+2}{x}  \, x^{m+n+1} 
 = (-1)^m \frac{\binom{n+\nu}{m+n+1}}{\binom{m+n}{m}} \, x^{m+n+1} + O(x^{m+n+2}).
 \end{aligned}
 \end{equation}
In a famous paper~\cite{BakerCube} in 1964 preceding his first general lower bounds on linear forms in logarithms of algebraic numbers, 
Baker followed the method of Thue and Siegel~\cite{SiegelThue} to establish the explicit Diophantine inequality $|\sqrt[3]{2} - p/q| > 10^{-6} q^{-2.955}$ by considering the sequence of rational
approximations to~$\sqrt[3]{2}$ obtained from the specializations $m=n$, $\nu = 1/3$, and $x := 3/128$ in~\cref{hyper poly}. A crucial point of the numerology
is that the linear forms thus specialized in $1$ and~$\sqrt[3]{1-3/128} = (5/8) \sqrt[3]{2}$ have exponentially small denominators, for instance of the form
$\binom{2n}{n} 2^{7n} 3^{\lfloor n/2 \rfloor}$.  More in line with the Ap\'ery limits method and our holonomy bounds, the generating function of 
those linear forms is a $G$-function in the sense of Siegel~\cite[\S~VII, page~43]{Siegel1929SNS}, obeying a certain third-order linear homogeneous ODE 
on the domain $\mathbb{P}^1 \setminus \left\{0, \infty, \big( (1 \pm \sqrt{1-3/128} \, )/ 2 \big)^{-2} \right\}$,
 overconvergent at the first singularity 
$ ( ( 1 + \sqrt{1-3/128} \, )/2 )^{-2} \approx 1.012$, and, as a result, holomorphic all the way up to the next singularity $ ( ( 1 - \sqrt{1-3/128} \, )/2 )^{-2}\approx 28784.766$. Chudnovsky's improvements in~\cite{ChudnovskyThueSiegel} came from  
a close study of which
 primes dividing $\binom{2n}{n}$ actually appear
in the denominator, and 
therefore to compute the exact asymptotics of the true denominator.
In more recent times, Bennett
has refined the hypergeometric method,  improving~\cite{BennettCubeTwo} Baker's example to the clean bound $|\sqrt[3]{2} - p/q| > q^{-2.5}/4$, and  establishing in collaboration with de Weger the striking theorem~\cite{BennettThueSiegel,BennettdeWeger}  that for every $r \geq 3$ and $a, b \in \Z$ the original Thue equation $aX^r - bY^r = 1$ has not more than one solution in positive integers $(X,Y) \in \N^2$
 (in particular, by bridging the gap to the wholly complementary range of applicability of the theory of linear forms in logarithms).

\subsection{Multivalence.} \label{sec:multivalent Thue}
The behavior of the generating functions of the $m=n$ diagonal specializations of the Hermite--Pad\'e forms~\cref{hyper poly}---namely, that they are holonomic and overconvergent at the smallest non-zero
singularity---exemplifies the mechanism of overconvergence for Ap\'ery limits described 
in~\cref{sec: irrational}. 
Following~\cite{L2chi}, we seek to continue the traditional hypergeometric method by considering the analytic continuations of these functions beyond their convergence discs.
Ideally, we can find 
 conformally large, possibly multivalent analytic mappings~$\varphi : (\Db,0) \to (\C,0)$, and then apply our holonomy bounds. To obtain better multivalence properties, it turns out to be advantageous to ignore the arithmetic denominator savings in~\cite{ChudnovskyThueSiegel}, and instead take the generating
 function of~\cref{hyper poly} after first clearing through the binomial coefficient.
 For   $\nu \in \Q$---as it turns out---we obtain an algebraic (dihedral) generating function on the domain
$
y \in
\mathbb{P}^1 \setminus \left\{ ( (1 \pm \sqrt{1-x} ) \, / x )^2, \infty \right\}
$ ($y = 0$ is no longer a singularity!). 
We have the clean hypergeometric polynomial generating  function identity 
$$
  \sum_{m,n=0}^{\infty} 
 \pFq{2}{1}{-\nu-n,-m}{-m-n}{x}  
 \binom{m+n}{m} \,  y^m z^n = \frac{(1 - x y)^{\nu}}{1 - y - z + xyz}.
 $$
On the $m=n$ diagonal, this gives, by a standard Cauchy integral (over a ``vanishing cycle'' \cite{Furstenberg,DeligneE}): 
 
 \begin{proposition}  \label{some dihedral algebra}
For $\nu \in \Q$ rational of denominator~$r$, the functions 
\begin{equation}
\begin{aligned} \label{Pade Chudnovsky norm}
A_{\nu}(x,y)  := \sum_{n=0}^{\infty}  \sum_{k=0}^{n} \left\{
\frac{(n-\nu)(n-1-\nu) \cdots (k+1-\nu)}{(n-k)!} \,  \binom{n+k}{k} \,  (-x)^{n-k}  \right\}  y^n  \\
= \sum_{n=0}^{\infty}   \pFq{2}{1}{-\nu-n,-n}{-2n}{x}  \binom{2n}{n} \, y^n  \in  \Z\llbracket y, xy/r^2 \rrbracket, \\
 B_{\nu}(x,y)  := A_{-\nu}(x,y) = \sum_{n=0}^{\infty}   \pFq{2}{1}{\nu-n,-n}{-2n}{x}  \binom{2n}{n} \, y^n \in  \Z\llbracket y, xy/r^2 \rrbracket
\end{aligned}
\end{equation}
are algebraic functions of~$y$ over~$\Q(x)$, with Galois group the spherical triangle group of signature $(2,2,r)$, that is~$D_r$, if~$r$ is odd, and $D_{r/2}$, if~$r$ is even.
They form a $\C$-basis of $y=0$ formal solutions to the second-order linear ODE
 of ``associated Legendre kind'' on $y \in \mathbb{P}^1 \setminus \left\{ ( (1 \pm \sqrt{1-x} \, )/x )^2, \infty \right\}$:
 \begin{equation} \label{dihedral equation}
\begin{aligned}
 (1 - 4 y+ 2 xy + x^2 y^2) \frac{d^2f}{dy^2} + 3 ( x^2 y + x - 2) \frac{df}{dy} + x^2(1-\nu^2)f = 0.
\end{aligned}
\end{equation}
Another $\C$-basis of solutions to~\cref{dihedral equation} are the basic dihedral functions 
\begin{equation} \label{companion pieces}
\frac{ \cos\left( \nu \, \arccos\frac{x^2 y + x - 2}{2\sqrt{1-x}}  \right) }{\sqrt{1 - 4 y + 2 x y + x^2 y^2}} \quad \textrm{ and } \quad \frac{ \sin\left(  \nu \, \arccos\frac{x^2 y + x - 2}{2\sqrt{1-x}}  \right) }{\sqrt{1 - 4 y + 2 xy + x^2 y^2}}. 
\end{equation}
For $|x| \leq 1$, the particular linear combination $ B_{\nu}(x,y)  - (1-x)^{\nu} A_{\nu}(x,y)$ spans the one-dimensional $\C$-vector space of solutions of~\cref{dihedral equation}
that are regular \emph{(``overconverge'')} at the 
singularity 
$y= ( (1 - \sqrt{1-x} \, )/x )^2$.
The functions $A_{\nu}(x,y), B_{\nu}(x,y)$ of~\cref{Pade Chudnovsky norm} have the 
property that, at
every finite place $v \in M_{\Q(x, \prod_{p \mid r-1} p^{1/(p-1)})}^{\mathrm{fin}}$ when~$x \in \Qb$ is algebraic, they map the closed $v$-adic disc $| y |_v \leq \max(1,|x|_v)^{-1} \, | r \prod_{p \mid r} p^{1/(p-1)} |_v$ to the closed unit disc $|Y|_v \leq 1$. 
 \end{proposition}

The local monodromies of~\cref{dihedral equation} are $\Z/2\Z$ at the finite singularities $y = ( (1 \pm \sqrt{1-x} \, )/x )^2$, 
and $\Z/r\Z$ at $y = \infty$. We take up David and Gregory Chudnovskys' arithmetic algebraization method~\cite{ChudnovskyAlg} for cases of Grothendieck's $p$-curvature conjecture 
(as extended by Andr\'e~\cite[\S\S~1.2--1.4]{AndreG} and Bost~\cite{BostFoliations}
when the monodromy group is {\it a priori} virtually solvable).
The exponential map to a commutative algebraic group  gives rise~\cite{SerreApp} to certain entire meromorphic resolutions $\varphi$ (generally in several variables) of growth order~$\leq 2$.
One then exploits these resolutions to prove  algebraicity  (under the running hypotheses)
 in a manner similar to that of~\cref{sec:quantitative}.
In our situation, we have an {\it a priori} algebraic function~$H = B - \eta A$ with holonomic rank~$2$, but one of a high algebraicity degree~$\in \{r, 2r\}$, with an algebraic Ap\'ery limit of the form $\eta = \sqrt[r]{a}$ (if we specialize $x := 1-a$ with $a \in \Qb$); and we can take the entire map~$\varphi \in \mathcal{O}(\C)$ to be holomorphic in just one single variable, and with growth order~$1$ (for $A$ and~$B$ simultaneously), 
or~$1/2$ (for the overconvergent combination $H := B - \eta A$). This pair of maps
is universal for the general (``spherical triangle type $(2,2,\infty)$'') ODE on $\mathbb{P}^1 \setminus \{\alpha, \beta, \infty\}$ having order~$2$ 
local monodromies around the two finite singularities~$\alpha$ and~$\beta$. 
The latter map---crucially for our application, of growth order strictly~$<1$---has never apparently been considered in this circle
of problems before.
 In these arrangements, even though both our functions as well as our Ap\'ery limits are {\it a priori} algebraic, 
 we continue to get striking conclusions when we apply the 
 Diophantine extensions in~\cref{sec:quantitative} of the holonomy bounds to the case of the system $\{1, H, H^2, \ldots, H^k\}$, for an optimal choice of the power~$k < r$. 
  We can  write down 
 explicit entire holomorphic maps~$\varphi$ that apply to every local system of ramification type $(2,2,\bullet)$ on $\mathbb{P}^1 \setminus \{\alpha,\beta,\infty\}$, as follows:
 
 \begin{lemma} \label{multivalent maps}
 Let~$\alpha \ne \beta$, and let~$f \in \C \llbracket x \rrbracket$ be a holomorphic function
 germ which analytically
 continues as a meromorphic function along all paths in $\mathbf{P}^1 \setminus \{\alpha,\beta,\infty\}$, 
 with monodromy of order~$2$ around both~$\alpha$ and $\beta$. 
 \begin{enumerate}
 \item  \label{logver}
 $\psi_{\alpha,\beta}^* f$ is meromorphic on~$\C$, where
 $\displaystyle{\psi_{\alpha,\beta}(z)= 
 \frac{\alpha+\beta}{2} \cdot \left(1 - \cosh \left( \frac{z}{\sqrt{\alpha \beta}} \right) \right)  +  \sqrt{\alpha \beta}  \cdot \sinh \left( \frac{z}{\sqrt{\alpha \beta}} \right)}
  =  z + \ldots $.
 \item  \label{overconvergence mapping}
 Suppose moreover that $f$ extends analytically at~$z=\alpha$  along some path~$\gamma$ from~$0$ to~$\alpha$.
Then, for some choice of~$u$ with
 $u := \displaystyle{\frac{1}{2}  \log{\frac{\sqrt{\beta}+\sqrt{\alpha}}{\sqrt{\beta}-\sqrt{\alpha}}}}$ and
 $\tanh^2(u) = \alpha/\beta$,  the pullback $\varphi_{\alpha,\beta}^* f$ is meromorphic on~$\C$, 
 where
\begin{equation}
\label{overconvergentfunction}
\varphi_{\alpha,\beta}(z) 
= \alpha  \left(1 - \frac{\sinh \left( \sqrt{ \displaystyle{u(u - z \tanh(u)/\alpha}} \right)^2 }
{\sinh^2(u)}\right) = z + \ldots 
\end{equation}
\end{enumerate}
The entire holomorphic maps~$\psi_{\alpha,\beta}$ and $\varphi_{\alpha,\beta}$ have growth orders~$1$ and~$1/2$, respectively. 
\end{lemma}

If $|\alpha| < |\beta|$, and the path~$\gamma$ 
 lies within the disc of radius~$|\alpha|$, then one can take~$u = \tanh^{-1}(\sqrt{\alpha/\beta})$
to be given by the corresponding convergent power series. 
\cref{multivalent maps} can be verified by an
elementary but unenlightening calculation, so instead we give some hints as to where these formulas
come from.
For the first claim, we can pull back via a quadratic map ramifying over~$\alpha$ and~$\beta$, and thus reduce
to the case~$\mathbf{P}^1 \setminus \{1,\infty\}$ with monodromy~$\Z$ and universal covering map~$\widetilde{\psi}(z) = 1 - e^{-z}$. For the second claim,
there is a  unique quadratic map~$h: \mathbf{P}^1 \rightarrow \mathbf{P}^1$ with~$h(z) = z + O(z^2)$ that ramifies over $\{\beta, \infty\}$, so~$h^{-1}(\infty) = \{\infty\}$ and
the order~$2$ singularity at~$\beta$ is resolved; explicitly, $h(z) =  z - z^2/(4 \beta)$. 
Let~$h^{-1}(\alpha)= \{\walpha,\wbeta\}$.
The overconvergence along~$\gamma$ pulls back
to an overconvergence along a path~$\wgamma$ from~$0$ to~$\walpha$ (which distinguishes~$\walpha$
from~$\wbeta$). Moreover, the monodromy of~$h^* f$ is once again of type~$(2,2,\bullet)$.
One can now successively iterate these quadratic maps, and the limit converges---in some neighborhood of~$z=0$---to~$\varphi_{\alpha,\beta}$ (where the choice of~$\gamma$ determines~$u$).

\subsection{An effective irrationality type for high order roots.}  \label{sec:raizes}
Whereas the $x=-1/n$ specialization to~$\sqrt[r]{1-x}$ in~\cref{hyper poly} directly constructs a sequence of rational approximations to~$\sqrt[r]{1+1/n}$ 
which converge fast enough to supply an effective irrationality measure $\mu \to 2$ in the limit $n \gg_r 1$, historically it was not easy to obtain
any nontrivial effective irrationality measure for~$\sqrt[r]{2}$ for large~$r$. This had to wait until~1971 with the first of Baker's {\it Sharpening} series of papers~\cite{BakerSharpeningOne}. The asymptotic bound $\mueff(\sqrt[r]{2}) \ll \log{r}$ in~\cite{BakerHighRoots} is, up to the absolute constant, still the state-of-art today. 
But \cref{theotrue} with the dihedral algebraic construction work together to directly supply a sub-Liouville effective bound $\mueff(\sqrt[r]{2}) \ll \sqrt{r(\log{r})^3}$. 
This irrationality measure is essentially of Siegel's strength, but it is effective, and applies just as well to any binomial algebraic number $\sqrt[r]{a/b}$.
 In a way, our argument effectivizes Mahler's explicit Hermite--Pad\'e treatment~\cite{MahlerThueSiegelschenSatzes} 
 of Siegel's~\cite{SiegelZeitschrift} ineffective exponent~$2\sqrt{r}$ for the particular
case of binomial numbers $\sqrt[r]{a/b}$ with $0 < a < b$.  Thue, Siegel, and Mahler worked with the explicit type~I Hermite--Pad\'e form directly for the extended binomial system $\{
1, (1-x)^{1/r}, \ldots, (1-x)^{k/r}\}$ and, with an optimal choice of $k < r$, specialized $x := 1-aq_0^r/bp_0^r \approx 0$ for a hypothetical exceptionally good ``anchoring'' approximation 
$p_0/q_0 \approx \sqrt[r]{a/b}$. By Baker's theorem, such a ``Thue ghost lever'' $p_0/q_0$
  is  known to \emph{not} exist, and this is the ultimate source of the ineffectivity in the whole Thue paradigm. 
  In contrast, we work with the 
``dual'' system $\{1, H, \ldots, H^k \}$ where $H$ is the dihedral algebraic generating function of the Hermite--Pad\'e forms of $\{1, (1-x)^{1/r}\}$, specialized 
directly at the requisite point $x = 1-a/b \in (0,1) \cap \Q$. The growth properties of the entire multivalent maps of~\cref{multivalent maps} (2) are restricted to a suitable choice of radius,
and input into~\cref{theomainreprise}.  The ensuing holonomy bound enables us to make a choice
of $k=m-1$ closely similar to Siegel's (and, by extension, Mahler's) optimizing choice, but now the Diophantine bound is effective as we have used the multivalence
of $\varphi_{\alpha,\beta}$, in place of the ghost Thue lever $p_0/q_0$, for 
increasing the overconvergence as far as possible.

Concretely in our $\sqrt[r]{2}$ example, consider $H(y) := B_{1/r}(-1,y) - 2^{1/r} A_{1/r}(-1,y) \in \Q\llbracket y \rrbracket + 2^{1/r} \Q\llbracket y \rrbracket$ from \cref{some dihedral algebra},
where the corresponding ODE has singularities at $\alpha=3-2 \sqrt{2}$ (the overconvergent singularity)
 and~$\beta = 3+2 \sqrt{2}$.
With $K = \Q$, we have $\bb = \mathbf{0}$, and the $p$-adic convergence radii have  $R_p = 1$, for $p \nmid r$, and $\prod_{p \mid r} R_p^{-1} = r \prod_{p \mid r} p^{1/(p-1)}$. 
Hence $\sum_{p \in M_{\Q}^{\mathrm{fin}}} \log{R_p} = -\log{r} - \sum_{p \mid r} \frac{\log{p}}{p-1} \geq -\log{r} - \log{\log{r}} -O(1)$. At the Archimedean place, we select
the map $\varphi_{\infty}(z) = \varphi(z) := \varphi_{\alpha,\beta}(Rz)$
where~$\varphi_{\alpha,\beta}$ is given by~\cref{overconvergentfunction}. Conditionally on $k< r$ and the positive denominator, and  making use
both of the  explicit formula~\cref{overconvergentfunction} and its 
 low asymptotic growth order of~$1/2$, our holonomy bound reads  
 \begin{equation} \label{sample holonomy quotient}
k+1 \ll \frac{\sqrt{R}}{\log{R} - \log{(r \log{r})}  - c_0 - (c_0 k/\mueff(\sqrt[r]{2}))\log{R} }. 
\end{equation}
The choices
$R := c_1 r \log{r}$ 
and  $k := c_2 \sqrt{r\log{r}}$ give an effective measure of irrationality $\mueff(\sqrt[r]{2}) \ll \sqrt{r(\log{r})^3}$. 

This computation extends directly to the $K$-relative irrationality measures of arbitrary roots $\sqrt[r]{a}$ from elements $a \in K$. 
But to get to~\cref{main Diophantine theorem} by Bombieri's elementary geometry of numbers reduction in~\cref{the new proof}, we need a bound that has Waldschmidt's feature of remaining asymptotically sub-Liouville in the uniformity range $r \gg 1+h(a)$,
as in Bombieri's~\cite[Theorem~1]{BombieriGm}. 
The best such bounds~\cite{Waldschmidt1993,Gouillon,Chim} have the form  $|\log{a} - r \log{b}|_u \gg_{p_v} - d^4 h \max(1, \log(r/h))$, with $d = [K:\Q]$, $h = 1+h(a)$, and a tower $u \mid v \mid p_v$ of places with
$u \in M_{K(\sqrt[r]{a})}$, $p_v \in M_{\Q}$. 
However, any effective bound of the shape 
$\gg - r /g(r/h)$ with $g(t) = g_{K,v}(t) \to \infty$ is sufficient.
(The papers~\cite{BombieriGm,BombieriCohenII,BombieriCohenElements} on the Thue--Siegel principle reach a~$g(t)$ proportional to 
$\sqrt{\log{t}}$, $t/(\log{t})^7$, and $ t/(\log{t})^5$, respectively.)
Our method with applying the arithmetic holonomy bounds to~\cref{some dihedral algebra} leads quite straightforwardly to a bound with $g(t) = t^{\frac{1}{2d}} (\log{t})^{-3}$, 
at least under the assumption --- entirely harmless for the application to~\cref{main Diophantine theorem} --- that~$r$ factorizes suitably. 

The above example and computation, contained in the next proposition as the special case $K= \Q$, $l = 1$, and~$a=2$, extends with Waldschmidt's uniformity as follows. 
In the following, we consider~$a \in K$ in our arbitrary number field~$K$. We write $d := [K:\Q]$, and we consider two alternatives on the place~$v \in M_K$.

\begin{proposition}[Archimedean case] \label{effective Thue-Siegel-Mahler}
Let $l,s \in \N$ and $r = ls$. For every place~$u \mid v$ of $K(\sqrt[r]{a})$, 
\begin{equation}   \label{full bound Archimedean}
\log{|\sqrt[r]{a} - b|_u} \gg_d  - \big(1 + h(a^{1/r}) + h(b)\big)  e^{ h(a)/l } \, l   s^{\frac{d}{d+1}} (\log{s})^{\frac{2d+1}{d+1}},
\end{equation}
holds for all~$b \in K$,
provided $s \geq c_3$. 
In particular, \begin{equation} \label{effective root}
\mueffKu(\sqrt[r]{a}) \ll_d   e^{ h(a)/l } \, l  s^{\frac{[K:\Q]}{[K:\Q]+1}} (\log{s})^{\frac{2[K:\Q]+1}{[K:\Q]+1}}.
\end{equation} 
\end{proposition}

\begin{proof}
We describe the basic idea, whose proper detail is deferred to a subsequent publication. 
By a standard reduction, we may replace~$a$ by $a^{\pm m}$ for some $1 \le m \le 6$
so that $|1 - a^{1/l}|_u \leq 1$. 
 Apply~\cref{some dihedral algebra} to the exponent $\nu := 1/s$ 
 and the specialization $x := 1 - a^{1/l}$ of height $h(a)/l + O(1)$. With $H := B_{1/s}(1-a^{1/l},y) - \sqrt[r]{a} \, A_{1/s}(1-a^{1/l},y) \in K[a^{1/l}]\llbracket y \rrbracket + \sqrt[r]{a} \, K[a^{1/l}]\llbracket y \rrbracket$, we apply~\cref{theotrue} to the system $\{1, H, \ldots, H^k\}$, changing the variable notation~$x$ of~\cref{sec:quantitative} to our present variable~$y$, and working over the number field~$K' := K(a^{1/l})$ to replace the field denoted~$K$ there. The extension~$K'/K$ has degree at most~$l$. At each place $w \in M_{K'}$
 extending~$v$, we select $\eta_w  := \zeta \sqrt[r]{a}$ for the $l$-th root of unity~$\zeta = \zeta_w$ for which the $w$-place of $ \zeta \sqrt[r]{a}$
 is equal to the $u$-place of $\sqrt[r]{a}$. 
  Here again, $k < s$ is a parameter to optimize. With $\alpha, \beta \in \C$ the dihedral ODE singularities, we take the map~$\varphi_w(z) := \varphi_{\alpha,\beta}(R^2z)$
  of~\cref{multivalent maps}, for~$w \mid v$, and the map $\varphi_w(z) := \psi_{\alpha,\beta}(R z)$, for the other Archimedean places~$w \nmid v$. This time the holonomy quotient~\cref{sample holonomy quotient} reads (upon a positive denominator):  
 \begin{equation}   \label{order 1/2 and holonomy quotients}
k+1 \ll \frac{R}{\frac{d+1}{d} \log{R} - \log{(s \log{s})} - c_0 -  h(a)/l- c_4 - (c_4 \, kl/\mueffKu(\sqrt[r]{a}))\log{R} }. 
\end{equation}
The similar choices of~$R$ and~$k$ as in the~$\sqrt[r]{2}$ illustrative example now lead directly to the bound~\eqref{effective root}. The uniform dependence on the height $h(a)$ 
is traced back to the dependence of the constant~$C$ in~\cref{rmk: effC}. 
\end{proof}

In the nonarchimedean case we have a closely similar bound, now based on an Archimedean multivalence in~\cref{multivalent maps} together with the special circumstance of Teichm\"uller representatives: $\lim_{n \to \infty} a^{p^n} \in \mu_{p^{\infty}}(K_v)$, if
$p$ is the residual characteristic of~$K_v$.  

\begin{proposition}[Nonarchimedean case]  \label{effective p-adic}
Let $p = p_v$ be the residual characteristic of~$v$ and $h \in [1,r/e]$ a parameter. Consider $r = ls$ with $l \equiv 1 \mod{p}$,  where $s \asymp \sqrt{r/h}$ and $l \asymp \sqrt{rh}$. Then, for all~$a \in K$ 
with $h(a) \leq h$, and at every place $u \mid v$ of~$K(\sqrt[r]{a})$, 
\begin{equation} \label{effective root too}
\mueffKu(\sqrt[r]{a}) \ll_{v,d}   h  \left( r/h \right)^{\frac{2d-1}{2d}} \log^2(r/h). 
\end{equation}
More precisely, for all~$b \in K$, the condition~$r \geq c_5h$ implies 
\begin{equation}   \label{full bound nonarchimedean}
\log{|\sqrt[r]{a} - b|_u} \gg_{v,d}  -\left(1 +  h(\sqrt[r]{a}) + h(b)\right)  \, h  \left( r/h \right)^{\frac{2d-1}{2d}} \log^2(r/h). 
\end{equation}
\end{proposition}

\begin{proof}
We indicate only a brief guideline. A standard reduction allows us to assume that $(1-a)/p$ is an algebraic $p$-integer. 
Write $q := p^n$ with $n$ a parameter, and apply~\cref{some dihedral algebra} to the exponent $\nu := 1/(qs)$ 
 and the specialization $x := 1 - a^{q/l}$, with $|x|_w \leq |q|_w$ for all~$w \mid p$, and of height $qh(a)/l + O(1)$. As 
 in~\cref{sec:zeta2(5)} below, the multivalence input will be Archimedean with a choice of radius~$R$, this time applied to  a
 map from~\cref{multivalent maps}~(1) at all Archimedean places. Now the holonomy quotient reads (upon a positive denominator):
  \begin{equation}  \label{hol too}
k+1 \ll_{v,d} \frac{R + d^{-1} \log{q}}{ \log{R} +  d^{-1} \log{q} - \log{(s\log{s})} - c_6  - qh(a)/l - (c_6 \, kl/\mueffKu(\sqrt[r]{a}))\log{q} }. 
\end{equation}
Select $q \asymp_{p} s \asymp \sqrt{{r/h}}$, so that~$qh(a)/l \ll 1$, and $R \asymp_{c_6} (s\log{s})/q^{1/d} \asymp s^{1-1/d} \log{s}$. With $k := \lfloor c_7 \, s^{1-1/d} \log{s} \rfloor$ the holonomy bound \cref{hol too} amounts to~\cref{effective root too}. The refinement to~\eqref{full bound nonarchimedean} is once again
routine from tracking the constant~$C$ in~\cref{rmk: effC}, as we will do in proper detail in a subsequent paper.  
 \end{proof}

\subsection{A new proof of the transcendence of~$\pi$ and other logarithms.}  \label{new proof of transcendency of Pi}
If we take the $\nu \to 0$ limit of~\cref{dihedral equation}, we obtain an ODE
whose solutions are given by $A(x,y) :=  \left( 1 - 4y + 2xy + x^2y^2 \right)^{-1/2}$ 
and
$B(x,y) := 
 2 A(x,y) \log{ \left(  1 - x \left(  1 + xy - \sqrt{ 1 - 4y + 2xy + x^2y^2} \right) /2 \right) }$, and
such that $H := B(x,y) - \log(1-x) A(x,y)$ is the $y$-generating series of the diagonal Hermite--Pad\'e approximants of the logarithm function~\cite[\S~3.3.7]{L2chi}. 
The domain $y \in \P^1 \setminus \{((1\pm \sqrt{1-x})/x),\infty\}$ and overconvergence property of~$H$ are exactly the same as in the discussion with the algebraic dihedral functions in~\cref{sec:multivalent Thue}, but now $H$ is a transcendental function with an infinite dihedral monodromy $D_{\infty} \cong (\Z/2\Z) \ast (\Z/2\Z)$,
and correspondingly~$B \in \Q[x] \llbracket y \rrbracket$ has~$[1,\ldots,n]$ type denominators. 
Examining the analog of the holonomy quotients~\cref{order 1/2 and holonomy quotients} and~\cref{hol too} 
leads in this context to a new proof of the transcendence of the complex and $p$-adic logarithms~$\log{a}$ at all algebraic arguments $a = 1-x \in \Qb^{\times} \setminus \{1\} 
\hookrightarrow \C$
or  $a \in \Qb^{\times} \setminus \mu_{p^{\infty}} \hookrightarrow \C_p$ with $|1-a|_p < 1$. We illustrate this on the case of~$\pi$ (with $x=2$). Replacing~$y$ by~$i z$, we apply the dihedral  method
to $H(z) := B(z) - \pi A(z)$  on $z \in \P^1 \setminus \{-1/2,1/2,\infty\}$, where:
\begin{equation} \label{piformula}
A(z) := (1-4z^2)^{-1/2} \in \Z\llbracket z \rrbracket, \quad B(z) := -2A(z) \arcsin(2z) 
\in \sum_{n \ge 0} \frac{z^n \, \Z}{[1,\ldots,n]}.
\end{equation}
 If~$\pi$ were contained by some 
number field~$K$ of some degree~$d$, the system $\{1, H, \ldots, H^k\}$ would fit~\cref{theotrue} with type $[1,\ldots,n][1,\ldots,n/2] \cdots [1,\ldots,n/k]$, which has
$\tau(\bb) = \sum_{j=1}^{k} 1/j = \log{k} + O(1)$. At this point, taking again the maps $\varphi_{-1/2,1/2}(R^2z)$ at one (the ``overconvergent'') embedding $K \hookrightarrow \C$
and $\psi_{-1/2,1/2}(Rz)$ at the~$d-1$ other complex embeddings, the---qualitative, for simplicity---holonomy quotient bound similar to~\cref{order 1/2 and holonomy quotients} becomes
$
k+1 \ll R \big/ \left( \frac{d+1}{d} \log{R} - \sum_{j=1}^{k} 1/j \right),
$
provided a positive denominator. With the choice of radius $R := (4k)^{\frac{d}{d+1}}$,
this yields a contradiction for $k \gg_d 1$.

The alternate specialization~$x=-1$ recovers exactly the approximations used
by~\cite{AlladiRobinson,Apery} to produce
an explicit irrationality measure for~$\log 2$. In contrast, the~$x=2$  specialization considered
above does
not even directly give a sequence of rational numbers converging to~$\pi$, but yet  
still leads to
the transcendence of~$\pi$, and even to explicit (if not impressive) irrationality measures such as $\mu(\pi) \leq 15.086$.

\section{Conclusion of the new proof of~\cref{main Diophantine theorem}.}   \label{the new proof}

In~\cref{effective Thue-Siegel-Mahler,effective p-adic} we constructed a  monotonic function $g = g_{K,v} : \R_{\geq 0} \to \R_{> 1}$ which can be made fully explicit,
depending only on~$K$ and~$v$ and having $\lim_{x \to \infty} g(x) = \infty$, such that, for all $r \in \N$ possessing a factor~$l \asymp 1+h(a)$, resp. a factor~$l \asymp \sqrt{r(1+h(a))}$ congruent to~$1\mod{p}$, the following lower bound in two logarithms holds uniformly for all places $u \mid v$ of~$K(\sqrt[r]{a})$ and all elements $a,\eta \in K$: 
\begin{equation}  \label{Waldschmidt uniformity abstraction}
\log{|1-a^{1/r}\eta|_u} \geq - \left(\log{2}+h(a^{1/r}) + h(\eta) \right) r \big/ g(r/h(a)). 
\end{equation}
Construct $c_8(K,v,\varepsilon) \in \R_{\geq 1}$
so that $g_{K,v}(x) \geq 2/\varepsilon$ for all $x \geq c_8(K,v,\varepsilon)$. Then, for all exponents~$r$ that meet our \emph{ad hoc} factorizability condition for~\eqref{Waldschmidt uniformity abstraction}, we derive: 
\begin{equation}   \label{binomial Diophantine}
a, \eta \in K \textrm{ with }  r \geq c_8(K,v,\varepsilon) h(a)  \quad \Longrightarrow \quad |1-a^{1/r} \eta|_u \geq
H(a)^{-\varepsilon/2} (2H(\eta))^{-\varepsilon r/2}. 
\end{equation}
The next type of argument to conclude~\cref{main Diophantine theorem}  is due to Bombieri~\cite[\S~8]{BombieriGm} and is now standard. The simplest form of~\cite[Lemma~4]{BombieriGm} is: 
\begin{lemma} \label{simplest lemma}
Consider $n_1, \ldots, n_t \in \Z$ and two 
integer parameters $Q, N \in \N$. There exists a positive integer~$r$ of the form
$r = Q!N/q$ for some~$q \in \{1, \ldots, Q\}$
and rational integers $p_1, \ldots, p_t \in \Z$ satisfying 
\begin{equation}  \label{Dirichlet approx}
|n_i - rp_i|  \leq rQ^{-1/t}, \quad i = 1, \ldots, t. 
\end{equation} 
\end{lemma}
\begin{proof}
Dirichlet's approximation theorem applied to the vector $(n_1/Q!N, \ldots, n_t/Q!N)$ and with~$Q$ as the Dirichlet parameter
gives a simultaneous rational approximation $(p_1/q, \ldots, p_t/q)$ with some common denominator $q \in \{1, \ldots, Q\}$ for which~\cref{Dirichlet approx} holds with~$r := Q!N/q \in \N$.
\end{proof}
\begin{remark} 
Of course, one can replace $Q!$ in this argument by the lowest common multiple $[1,\ldots,Q]$. Bombieri and Cohen used a slightly more involved version~\cite[Lemma~4]{BombieriGm}, \cite[\S~6]{BombieriCohenII} with weights, 
which for large $h(A)$ leads to the bounds with~\cref{depend C}. We suppress this feature for our introductory purposes in this paper. 
\end{remark}

Consider an arbitrary $\gamma = \zeta \xi_1^{n_1} \cdots \xi_t^{n_t}$ as in \cref{main Diophantine theorem}, where $\zeta \in \Gamma_{\mathrm{tors}}$
is a root of unity and $\xi_1, \ldots, \xi_t$ are generators for the free part $\Gamma / \Gamma_{\mathrm{tors}}$. To bound the height of~$\gamma$ means 
to bound the largest absolute value of an exponent~$n_i$. With the~$r$ and the~$p_1, \ldots, p_t$ coming from \cref{simplest lemma},
there exists a decomposition 
$$
\gamma = a_0 \eta^r, \quad \eta := \xi_1^{p_1} \cdots \xi_t^{p_t}, \quad a_0 := \zeta \xi_1^{n_1 - rp_1} \cdots \xi_t^{n_t - rp_t}, \quad \text{with} \quad  h(a_0) \leq c_9(\Gamma)  \,  rQ^{-1/t}. 
$$ 
We arrive at a coset representative $a := Aa_0 \in A\gamma \Gamma^r < A\Gamma$ with height satisfying
\begin{equation} \label{element construction}
|h(a) - h(A)| \leq h(a_0)  \leq c_9(\Gamma)  \,  rQ^{-1/t}.
\end{equation}
In~\cref{simplest lemma}, we select 
\begin{equation}  \label{choice of Q}
Q = Q(K,v,\Gamma,\varepsilon) :=  \lceil (2c_8(K,v,\varepsilon))^t c_9(\Gamma)^t \rceil,
\end{equation}
getting by~\eqref{element construction} the bound
\(
h(a) \leq h(A) + r / \left( 2 c_8(K,v,\varepsilon) \right).
\)
Upon changing the coset representative~$A$ of~$A\Gamma$, for the purpose of proving~\cref{main Diophantine theorem} we may and do assume that~$h(A) \geq L$, with another and arbitrary constructive function  $L : = L(K,v,\Gamma,\varepsilon)$
of the data which we are free to select later.
Choose $\alpha := \sqrt[r]{a}$ in a place $u \mid v$ of $L := K(\sqrt[r]{a})$ to be the branch closest to the $K$-rational element~$\eta^{-1} \in \Gamma \subset K$. 
Then $(\alpha \eta)^r = A\gamma$.
To supply the requisite exponent factorization $r = ls$, we will take the parameter~$N$ in~\cref{simplest lemma} to be a prime power $N = g^m$, where $g \in \{2,3\}$ is~$2$ unless~$v$ is nonarchimedean of residual characteristic~$2$, and~$g=3$ in the latter case. {\it Define~$N$ to be the smallest power of~$g$ that exceeds $2c_8(K,v,\varepsilon) h(A)/(Q-1)!$. }
Our remark regarding $h(A) > L$ then allows us to assume $N > Q!^2$. 
This in turn ensures that $r = NQ!/q$ does indeed possess both requisite 
factorizations.
The bound~\cref{Waldschmidt uniformity abstraction} is therefore in place, and~\cref{binomial Diophantine}
follows. 

Our (or rather Bombieri's) choices of~$Q$ and~$N$ were selected precisely to supply the condition $r \geq c_8(K,v,\varepsilon) h(a)$ for~\cref{binomial Diophantine}. First, 
\cref{element construction} with the choice~\cref{choice of Q} give $h(a) \leq h(A) + r / \left( 2 c_8(K,v,\varepsilon) \right)$. Then the requisite bound $r \geq c_8(K,v,\varepsilon) h(a)$
follows from the implication
\begin{equation}
r \geq (Q-1)!N \geq 2c_8(K,v,\varepsilon) h(A)  \quad \Longrightarrow \quad r \geq
c_8(K,v,\varepsilon) \left(  h(A) + r / \left( 2 c_8(K,v,\varepsilon) \right) \right) 
\geq c_8(K,v,\varepsilon) h(a),
\end{equation}

In the direction opposite to~\eqref{binomial Diophantine}, as in~\cite[Equation (8.4)]{BombieriGm}, we have
\begin{equation}
|1-\alpha \eta|_u  < 2^r |1 - A\gamma|_v \leq 2^r H(\gamma)^{-\varepsilon} \leq 2^{Q!N}H(A)^{\varepsilon} H(A\gamma)^{-\varepsilon} = 2^{Q!N} H(A)^{\varepsilon} H(\alpha \eta)^{-\varepsilon r}, 
\end{equation}
under our assumption~\cref{main Diophantine inequality} in the theorem-under-proof. Together these conditions bound $h(\alpha\eta) \leq c_{10}(K,v,\Gamma,\varepsilon)$, 
therefore $h(A\gamma) \leq c_{10}(K,v,\Gamma,\varepsilon) r \leq c_{10}(K,v,\Gamma,\varepsilon) Q!N$. Hence ultimately
$h(\gamma) \leq h(A) + h(A\gamma) \leq h(A) +  c_{10}(K,v,\Gamma,\varepsilon)  Q!N
 \leq c_{11}(K,v,\Gamma,\varepsilon)( 1 + h(A))$ due to the choice of~$N$. We arrive at $ C(K,v,\Gamma,\varepsilon) := c_{11}(K,v,\Gamma,\varepsilon)$.

\section{Irrationality measures: Examples.} We finish with two illustrative examples of~\cref{theotrue} and its variants. The first is an irrationality
measure for~$L(2,\chi_{-3})$, making one of the main irrationality results of~\cite{L2chi}
effective (see~\cref{rmk: effC}). The second is the irrationality of~$\zeta_2(5) \in \Q_2$, which in early~2020 had been our first joint result in this collaboration. 

\subsection{An irrationality measure for~$L(2,\chi_{-3})$.}
We can apply~\cref{theomainreprise} (in the form~\cref{withintegrationsremark})
to give an explicit irrationality measure for~$L(2,\chi_{-3})$ using
the explicit template~$\varphi(z) = h(\psi(z))$ of~\cite[Def~A.4.2]{L2chi}.
The overconvergent singularity in this setting corresponds to~$-1/72$.
The unique preimage~$x$ of~$-1/72$ of~$h$ in the image of~$\psi(z)$ is computed in
 (\cite[Lemma~A.4.4]{L2chi}); the~$\rho$ with~$\psi(\rho)=x$
 and thus~$\varphi(\rho) = -1/72$
is
$$\rho \sim   0.000049668  - 0.000070341 i, \quad |\rho^{-1}| < 11614.$$
Precisely~$7$ of our~$m=14$ functions are genuine (Lemma~\cite[12.1.1]{L2chi}), 
and so~$\gamma = 1/2$.
Thus~\cref{withintegrations} becomes
\begin{equation}
\label{lessthan14}
14 < 
\frac{
11.845
}
{
\displaystyle{\log \left( 256  \cdot
\frac{5448339453535586608000000000}{8658833407565631122430056127}
\right)
- \left(\frac{27}{80}  + \frac{191}{49}\right)
-  \log(11614)  \left(\frac{\kappa - (1/4)}{\kappa^2}\right)
}
}
\end{equation}
As~$\kappa \rightarrow \infty$, the RHS converges to $13.9938\ldots < 14$
(cf.~\cite[A.5.1]{L2chi}), which
shows that~$L(2,\chi_{-3}) \notin \Q$. From~\cref{lessthan14},
 we deduce a contradiction
as soon as~$\kappa$ is sufficiently large. Explicitly, we deduce:

\begin{theorem} For all $p/q \in \Q$ except possibly a finite and computable list, we have
$\displaystyle{\left| L(2,\chi_{-3}) - \frac{p}{q} \right| > \frac{1}{q^{24781}}}$.
\end{theorem} 

The same theorem (and proof) holds with $L(2,\chi_{-3})$ replaced by any~$\Q$-linear
combination of~$L(2,\chi_{-3})$ and~$\pi^2$.

\subsection{The irrationality of $\zeta_2(5)\in \Q_2$.} \label{sec:zeta2(5)}

Let~$\WP$ be the space of continuous homomorphisms
$\kappa: \Z^{\times}_p \rightarrow \C^{\times}_p$ with~$\kappa(-1) = 1$.
The $p$-adic zeta function $\zeta_p(\kappa)$ is a rigid analytic function on~$\WP$
away from the point~$\kappa = 1$, at which it has a simple pole.
If~$\kappa$ has the form~$x \mapsto \chi(x) x^{k}$ for an integer~$k$, then
$\zeta_p(\kappa) = L_p(\chi,1-k)$,
where~$L_p(\chi,k)$ is the Kubota--Leopoldt $p$-adic~$L$-function~\cite{Iwasawa}.
If~$\kappa$ has the form~$x \mapsto x^{2k}$ for an 
 integer~$k$, then we write~$\zeta_p(1+2k):=\zeta_p(\kappa)$.  For  \emph{complex}~$s$, write $\zeta^*(s) = (1 - p^{-s}) \zeta(s)$ for the $p$-deprived zeta function,
 that is, the usual Riemann zeta function with the local factor at~$p$ removed.
For a negative integer~$k$, there is an equality 
\begin{equation} \label{classical}
\zeta_p(1+2k) = \zeta^*(1+2k)
\end{equation}
which should be interpreted as meaning that both values (\emph{a priori} elements of~$\C_p$
and~$\C$ respectively) are both inside~$\Q$ and are equal.
For a \emph{positive} integer~$k$, the continuity of~$\zeta_p(\kappa)$
together with~\cref{classical} implies the equality
\begin{equation} \label{limitdef}
\zeta_p(1+2k)  = \lim_{m \rightarrow k} \zeta^*(1+2m),
\end{equation}
where the limit is is over $m \in \Z$ subject to the condition that $2m \equiv 2k \bmod p-1$
and the topology on~$\Z$ is induced from the inclusion~$\Z \subset \Z_p$.
One can view~\cref{limitdef} as interpolating the Kummer congruences on the Bernoulli
numbers.
The characterization of~$\zeta_p(1+2k)$ relevant for us is that, for any $2k \in \Z \setminus \{0\}$,
it is
 the unique $p$-adic number for which the formal power series
\begin{equation} \label{dontrepeat}
E^*_{-2k} := \frac{\zeta_p(1+2k)}{2} +  \sum_{n=1}^{\infty} \left( \sum_{ \substack{ d \mid n, \, p \nmid d } } d^{-2k-1} \right) q^n \in \Q_p  + q \, \Q \llbracket q \rrbracket
\end{equation}
is the $q$-expansion of an overconvergent
 $p$-adic modular form of level $\Gamma_0(p)$ and weight~$-2k$
 (\cite[\S2.2]{MR1696469}).  The specializations $E^*_{2k}$ when $k$ is a \emph{positive}
 integer are just the classical $p$-stabilized Eisenstein series of level~$\Gamma_0(p)$; they
 have coefficients in~$\Q$ by~\cref{classical}.
 We now prove that the~$2$-adic zeta value~$\zeta_2(5) \in \Q_2$ is irrational.

\begin{theorem}  \label{old2adic}
For sufficiently large $p$ or~$q$,  we have
$\displaystyle{\left| \zeta_2(5)- \frac{p}{q} \right|_2 > \frac{1}{\max(|p|,|q|)^{20}}}$.
\end{theorem}

\begin{proof}
We begin with the same setup as~\cite{Calegari}, where it is proved
that~$\zeta_2(3) \notin \Q$. 
The new input is to consider a non-trivial 
template~$\varphi(z)$ at the Archimedean place.
We have an identification $X_0(2) \cong \mathbb{P}^1_x$ with the Hauptmodul
\begin{equation} \label{haup}
x = x(q) := \frac{\Delta(2\tau)}{\Delta(\tau)} = q \prod_{n=1}^{\infty}(1+q^n)^{24},
\end{equation}
where the ordinary locus of~$X_0(2)$ is given by the two components~$|x|_2 \le 1$
and~$|x|_2 \ge 2^{12}$ and the complementary annulus is the supersingular locus.
We  may formally invert~\cref{haup} to obtain the expansion
$$
q = x - 24x^2 +852x^3 -35744x^4 + \cdots \in x + x^2\Z\llbracket x\rrbracket.
$$
We now consider (for~$p=2$) the following pair of Eisenstein series~\cref{dontrepeat}:
\begin{equation*}
E_{2k}^*  
  \in \Q + x \Z\llbracket x\rrbracket, \qquad
E^*_{-2k} \in \frac{\zeta_2(1+2k)}{2} + x\Q\llbracket x\rrbracket.
\end{equation*}
We have $\tau(E_{2k}^*(x)) = 0$ in the $x$-coordinate and
$\tau(E_{-2k}'(x)) = 2k+1$ in the $x$-coordinate, where~$E_{-2k}' \in x\Q\llbracket x\rrbracket$
is $E^*_{-2k}$ with the constant term omitted.
Suppose for contradiction that $\zeta_2(2k+1) \in \Q$,
so that $E^*_{-2k} \in \Q\llbracket x\rrbracket$.  It follows that the overconvergent modular \emph{function}  $H := E_{2k}^* E^*_{-2k}$
for $\Gamma_0(2)$
\emph{has rational coefficients} and  satisfies $\tau(H) = 2k+1$. One finds that
 radii $R_p = 1$ for $p \notin \{2,\infty\}$.
For $p = 2$, Buzzard's theorem~\cite[Thm~5.2]{buzzard} gives analyticity of
$E^*_{2k}$ and thus of~$H$ on
the union of the component of the ordinary locus containing~$\infty$ and
the entire supersingular locus. This region is precisely the disk $|x|_2 < R_2$ of radius $R_2 = 2^{12}$: hence this is (a lower bound for)  the $2$-adic convergence radius of the power series $H(x)$.
The Archimedean convergence radius of the power series $H(x)$
is $1/64$.  For~$k=1$, the usual Ap\'{e}ry argument gives that~$\zeta_2(3) \notin \Q$, and this
is essentially the content of~\cite{Calegari}.
For~$k = 2$, it is crucial that we exploit the analytic continuation of~$H(x)$
to a  region beyond this disc.
In particular, $H(x) = H(x(q))$ is analytic on the larger region $\{|q| < 1\}$, which has conformal radius~$64$
times the na\"{\i}ve disc in the coordinate~$x$.
As in~\cite[\S A]{L2chi}, the goal is to choose a simply connected region~$\Omega \subset D(0,1)$ so that~$x(q)$ restricted to~$\Omega$ is not too large and yet the conformal
radius of~$\Omega$ is not too small. 
Without trying too hard to optimize~$\Omega$, we
 initially choose~$\Omega$ 
 to be the circle:
$\displaystyle{\Omega = \Omegacirc  = \left\{z \in D(0,1),  \  \left|z + \frac{2}{5} \right| \le \frac{3}{5}  \right\}}$,
and thus (composing~$x(q)$ with the standard uniformization~$D(0,1) \rightarrow \Omegacirc$ sending~$0$
to~$0$) the choice
$$\varphi(z) =  x \left(  \frac{z}{2z+3}    \right)
= \frac{z}{2z+3} \prod_{n=1}^{\infty}\left(1+\left( \frac{z}{2z+3} \right)^n\right)^{24}.
  $$
with $|\varphi'(0)| = 1/3$.
By Kontsevich and Zagier (\cite{kontsevichzagier},
the proof of
Fact~1 in section~2.3), the power series $H(x)$ satisfies a linear inhomogeneous differential equation
over $\Q(x)$ of the exact minimal order $5 = 2k+1$, so that the functions $1, H(x), H'(x), H''(x), H'''(x)$ and
$H^{(iv)}(x)$ exhibit \emph{six} $\Q(x)$-linearly independent elements of the space  $\mathcal{V}\left(\varphi_2(z) = 2^{-12}z, \, \varphi_{\infty}(z) = \varphi(z); \vec{(5)} \right)$.
With~$R_p = 1$ for~$p \ne 2$ and~$R_{2} = 2^{12}$, and with
$r= 1$, $\bb = (0,5,5,5,5,5)$ and
$\displaystyle{\tau(\bb) = \frac{1}{36} \sum_{i=2}^{6} 5(2i-1) = \frac{35}{36} \cdot 5}$,
 our holonomy bound~\cref{hol_bound} implies that
\begin{equation} \label{fullforce}
\begin{aligned}
6 \le m  & \leq \frac{ \iint_{\T^2} \log|\varphi(z)-\varphi(w)|  + \sum_{p} \log R_p }{\log |\varphi'(0)| + \sum_{p} \log R_p  - \tau(\bb)}
= 
 \frac{2.13322 \ldots + 12 \log 2}{- \log 3 + 12 \log 2 -  175/36}  
= 4.43206 \ldots < 6, \end{aligned}
\end{equation}
which is the desired contradiction. We now turn to the irrationality measure.
Of the~$m=6$ functions, only one 
 lies in~$\Q \llbracket x \rrbracket$. 
Hence~$\gamma = 1/6$. With~$\rho_2=2^{-12}$,
we get an irrationality measure for any $\kappa$
satisfying (cf. \cref{fullforce})
\begin{equation} \label{weakermeasure}
\begin{aligned}
6 
& >  \frac{2.13322 \ldots + 12 \log 2}{- \log 3 + 12 \log 2 -  175/36 -12   \log(2) \frac{2(1-\gamma)\kappa - (1-\gamma)^2}{\kappa^2}}, \end{aligned}
\end{equation}
or $\kappa \sim 22.0724$. This can be improved with a more involved choice of~$\Omega$. 
For example, consider the alternate choice
 $\varphi(z) = x(\psi(z))$ with the lune contour~\cite[(A.1.1)]{L2chi}
 $\psi: D(0,1) \iso  \Omegalune \subset \overline{D(0,2/3)}$ given explicitly by:
 $$
 \psi(z) = \frac{2}{3} h\left(-z,\frac{2}{5}\right) =\frac{2}{3} \cdot \frac{(5^2+2^2)}{2(5^2-2^2)} \left(1 + z - \sqrt{1 - \frac{2(5^4 - 6 \cdot 5^2 2^2 + 2^4)}{(5^2+2^2)^2} z + z^2} \right) =  \frac{14}{29} \cdot z + \ldots$$
This leads to the estimate $\kappa < 19.7439 < 20$, where the analogue of~\cref{weakermeasure} is
 $$
6  < \frac{3.92881 \ldots + 12 \log 2}{(\log 14 - \log 29)  + 12 \log 2 -  175/36 -12   \log(2) \frac{2(1-\gamma)\kappa - (1-\gamma)^2}{\kappa^2}}.
$$
We  bound the value of the Bost--Charles integral above by the rearrangement
 integral~\cite[Prop~8.1.13]{L2chi} which is easier to compute rigorously.
\end{proof}

\begin{remark} \label{max} As far as the qualitative irrationality in~\cref{old2adic} is concerned, the proof of $\zeta_2(5) \notin \Q$ neither requires the full force
of the holonomy bounds in~\cite{L2chi}, nor even of their precursors in~\cite{UDC}. 
Indeed, the proof goes through   (with the same $\varphi_{\infty}$ corresponding
to~$\Omegacirc$)
 even after replacing the Bost--Charles
integral term in the numerator by $ \sup_{\T} \log|\varphi|$, 
and~$\tau(\bb)$ by~$\tau = 5$.  
The corresponding
computation for~\cref{fullforce} is: 
\begin{equation} \label{brutforce}
6 \le m   \leq \frac{  \max \log|\varphi(z)|  + \sum_{p} \log R_p }{\log |\varphi'(0)| + \sum_{p} \log R_p  - \tau}
=  \frac{\log \left((1/5) \prod_{n=1}^{\infty} (1 + (1/5)^n)^{24}\right) + 12 \log 2}{- \log 3 + 12 \log 2 -  5}  
= 5.52667 \ldots < 6.
\end{equation}
 We  proved
such bounds (together with the application to~$\zeta_2(5)$)  in 2020
(see~\cite[\S 15.7]{L2chi}), but we put
the argument aside in order to work towards the main theorems of~\cite{L2chi}; we apologize
for the delay!
Quite recently, 
Lai, Sprang, and Zudilin~\cite{LSZ} have found a completely new proof 
of the irrationality of~$\zeta_2(5)$, more in the spirit of Ap\'{e}ry's original
argument, by working  with a differential equation with better convergence properties. The corresponding local system in their case appears to have a genuinely different
geometric origin, coming from a moduli
space of Calabi--Yau $3$-folds rather than a symmetric power of the Picard--Fuchs equation
of a modular curve.
Neither of these results, however, have any direct bearing on the (ir)rationality
of the classical~$\zeta(5)$.\footnote{
One may wonder
why one can prove the irrationality of~$\zeta_2(5) \in \Q_2$ but not the irrationality of~$\zeta(5) \in \R$.
We give a heuristic explanation why the first problem is easier. Let~$\MT(\Z)$ denote the Tannakian category of mixed Tate motives over~$\Z$.
Associated to~$\MT(\Z)$ is a~$\Q$-algebra of real (fixed by complex conjugation) motivic periods~$\PP$~\cite{Brown}.
The comparison theorem between Betti cohomology and de Rham cohomology (over~$\C$) induces a period map \(\PP \rightarrow \R\),
and Grothendieck's period conjecture implies that this map is injective~\cite{MR2588609}.
A consequence of this conjecture is that the~$\Q$-algebra
generated by~$\zeta(2)$ and the odd values~$\zeta(2k+1)$ for~$k \ge 1$
is free on these generators; equivalently, that these values are all
transcendental and moreover algebraically independent.
What is perhaps surprising at first is that the~$p$-adic analogue of this story
behaves somewhat differently~\cite{Besser,Furusho}.
One of the most basic manifestations of this difference is
that the~$p$-adic comparison theorem naturally places the~$p$-adic
analogue of~$2 \pi i$ as an element in~$\BdR$ which maps to zero in~$\mathbf{C}_p$.
In particular, the only sensible definition of~$\zeta_p(2)$ 
returns~$\zeta_p(2) = 0$ (\cite[Rem~3.7]{YamashitaOne}). One nonetheless conjectures~\cite{Yamashita}
that the~$\Q$-algebra~$\Q(\zeta_p(3),\zeta_p(5),\zeta_p(7),\ldots)$
is free (on the obvious generators).
Let us consider
how the vanishing of~$\zeta_p(2)$ affects approaches to irrationality.
The first step in any Ap\'{e}ry type scheme is to produce a motivic local system
where the corresponding~$L$-value arises as a period.
 One issue, when  trying to construct such local systems
 is that various periods get ``mixed'' together and it is hard to separate
 the period one is particularly interested in.
 (One interesting discussion of these issues 
 from a geometric point of view 
 can be found in~\cite{BrownZudilin}.)
The point is now that, $p$-adically, the vanishing of~$\zeta_p(2)$
makes it easier to realize the odd zeta values  on their own.
Relevantly for us, the construction in~\cite{Calegari}
directly generalizes to produce an ODE over~$\mathbf{P}^1 \setminus \{0,-1/2^6,\infty\}$
associated to~$\zeta_2(2k+1)$ for any~$k$; we know of no simple analogous construction
for~$\zeta(2k+1)\in \R$. 
As another example; the ODE for the~$2$-adic Catalan's
constant in~\cite{Calegari} has the form~$\LL F = 0$ over~$\mathbf{P}^1 \setminus \{0,-1/2^4,\infty\}$ where
$$\LL = x(1+16 x)^2   \frac{d^2}{dx^2} + (1 + 16 x)^2  \frac{d}{dx}  - 4.$$
Up to scaling this is a hypergeometric equation which one can solve explicitly, namely, with:
\begin{equation}
\label{catalan}
A(x) =  (1 + 16 x)^{1/2} \cdot 
 \pFq{2}{1}{1/2,1/2}{1}{-16 x},\qquad
 C(x) = 
(1 + 16 x)^{1/2}  \cdot 
 \pFq{2}{1}{1/2,1/2}{1}{1 + 16 x}.
\end{equation}
The holomorphic solution~$B(x) \in \Q \llbracket x \rrbracket$ to the non-homogenous equation~$\LL F = 1+16 x$
with~$B(0) = 0$ and~$B'(0) = 1$ is of denominator type~$\tau = [1,2,\ldots,n]^2$.
Another solution to this equation~$\LL F=1+16x$ is given explicitly by
$$D(x) =  -\frac{1+16x}{4} \cdot \pFq{3}{2}{1,1,1}{3/2,3/2}{1+16 x}.$$
Over the real numbers, $D(x)$ is singular at~$x=0$, and we have the equality
  \begin{equation}
\label{real}
D(x) = B(x) - \frac{G}{2} \cdot A(x) - \frac{\pi^2}{16} \cdot C(x),
\end{equation}
where~$G \in \R$ is now the usual Catalan's constant. 
In contrast, $2$-adically, 
we see that~$D(x)$ converges for~$|1+16x|_2 < 2^{4}$,
or for~$|x|_2 < 2^{8}$, and so it is certainly regular at~$x=0$. Now the $2$-adic avatar of~\cref{real} simplifies to the formula
\begin{equation}
\label{2adic}
D(x) = B(x) - \frac{G_2}{2} \cdot A(x),
\end{equation}
in which~$G_2 \in \Q_2$ is the~$2$-adic Catalan constant. Thus~\cref{2adic}
 realizes~$G_2$ as an Ap\'{e}ry limit.
The key point is now that~\cref{real} is 
 precisely the ``same'' equation as~\cref{2adic} as long as we
 interpret the~$2$-adic meaning of~$(2 \pi i)^2$ to be zero, but Ap\'{e}ry's
 method only allows one to deduce that~$G_2 \in \Q_2 \setminus \Q$, but not that~$G \in \R \setminus \Q$.
That said, while this is an argument to explain why the $p$-adic zeta
values of small weight should be easier to study than their Archimedean counterparts,
there is also the contrasting remark that we can't as yet even rule
out the possibility that~$\zeta_p(3)=0$ for all sufficiently large~$p$,
let alone prove that all those values are irrational.
}
\end{remark}

\section*{Acknowledgments.}
As discussed in~\cite{L2chi}, our work on arithmetic holonomy bounds owes an obvious debt to the mathematics of the Chudnovskys~\cite{ChudnovskyAlg,ChudnovskyThueSiegel}, 
 Andr\'e \cite{AndreG}, Bost~\cite{BostFoliations}, and Bost--Charles \cite{BostCharles}.

\bibliographystyle{siamplain}
\bibliography{ICM}
\end{document}